\documentclass[11pt]{article}

\usepackage[a4paper,margin=1in]{geometry}
\usepackage{amsmath,amssymb,amsthm,mathtools}
\usepackage{mathrsfs}
\usepackage{hyperref}
\usepackage{enumitem}
\usepackage{microtype}

\usepackage{graphicx}
\usepackage{subfigure}
\usepackage{float}

\usepackage{multicol}
\usepackage{fancybox}

\usepackage{tabularx}
\usepackage{makecell}

\usepackage{oubraces}

\newcommand\citeS[1]{\cite{#1}}
\makeatletter
\newcommand{\citet}[2][]{%
  \if\relax\detokenize{#1}\relax
    \cite{#2}%
  \else
    \cite[#1]{#2}%
  \fi
}
\makeatother

\newcommand{\beq}{\begin{equation}}
  \newcommand{\eeq}{\end{equation}}

\hypersetup{colorlinks=true, linkcolor=blue, citecolor=blue, urlcolor=blue}

\newcommand{\R}{\mathcal{R}}
\newcommand{\C}{\mathbb{C}}
\newcommand{\Rp}{\mathbb{R}_+}

\def\RR{\mathbb{R}}
\def\N{\mathbb{N}}
\def\Q{\mathbb{Q}}
\def\Z{\mathbb{Z}}

\newcommand{\thetadeux}{y}

\newcommand{\G}{\mathcal{D}_\mathcal{R}}

\newtheorem{thm}{Theorem}[section]
\newtheorem{theorem}[thm]{Theorem}

\newtheorem{prop}[thm]{Proposition}
\newtheorem{corfr}[thm]{Corollary}

\newtheorem{lem}[thm]{Lemma}
\newtheorem{lemm}[thm]{Lemma}
\newtheorem{lemme}[thm]{Lemma}
\theoremstyle{definition}

\newtheorem{deffr}[thm]{Definition}
\newtheorem{rem}[thm]{Remark}
\newtheorem{exe}[thm]{Example}
\newtheorem{remarque}[thm]{Remark}
\newtheorem*{remarque*}{Remark}

\renewcommand{\geq}{\geqslant}
\renewcommand{\leq}{\leqslant}

\title{How to Study Reflected Brownian Motion in a Quadrant via Kernel Functional Equations?\\
\medskip
\large A short survey}
\author{Sandro Franceschi\thanks{SAMOVAR, Télécom SudParis, Institut Polytechnique de Paris, Palaiseau, France}}

\date{}

\begin{document}
\maketitle

\begin{abstract}
We survey a line of works studying semimartingale reflected Brownian motion (SRBM) in a quadrant (or, more generally, a wedge), covering both the non-degenerate setting (full-rank covariance) and degenerate setting (rank-one covariance).
Two main situations are emphasized: the recurrent case, where an invariant (stationary) measure exists, and the transient case, where the central objects are Green’s functions (potential measures). These measures typically arise from Kolmogorov forward (Fokker–Planck-type) equations. 
Beyond these quantities, for transient or killed models one is also interested in the Martin boundary of the process and, consequently, in all positive harmonic functions, which satisfy Kolmogorov backward equations. Depending on the geometry and parameters of the model, these harmonic functions often admit probabilistic interpretations in terms of absorption, escape, or drift to infinity (for instance, along an axis). 

All these measures and functions are studied through \emph{kernel functional equations} satisfied by their Laplace transforms. 
Several ways of solving these equations are reviewed, each leading to different types of results. Following the analytic approach developed for quarter-plane random walks by Fayolle, Iasnogorodski and Malyshev, a key preliminary step is the analytic continuation of the relevant Laplace transforms onto the complex algebraic curve defined by the zero set of the kernel. Carleman boundary value problem (BVP) techniques then yield explicit contour-integral representations for the Laplace transforms. In special parameter regimes, Tutte’s invariant method provides integral-free formulas and a sharp classification of the transforms according to their algebraic/differential complexity (rational, algebraic, D-finite, or differentially algebraic). Complex-analytic singularity analysis combined with saddle-point methods carried out on the kernel’s associated algebraic curve, produces precise two-dimensional asymptotics for the densities and functions under consideration. Finally, in the degenerate setting, the compensation approach provides an alternative constructive method, allowing one to build the densities as an infinite series through explicit iterative corrections adapted to the parabolic kernel geometry.
\end{abstract}

\tableofcontents


%
%
%

\newpage
\section{Introduction}

Over the past years, significant progress has been made in the explicit resolution of quantities associated with reflected Brownian motion in cones and orthants. While the stochastic analysis of such processes is inherently complex, a powerful and unifying perspective has emerged: many key objects of interest (invariant measures, Green functions, harmonic functions) can be characterized through \emph{kernel functional equations}. These equations reveal an underlying integrable structure, which enables the use of a broad spectrum of techniques ranging from complex analysis to algebraic and combinatorial methods.

This approach has been developed in detail in the author’s habilitation thesis \cite{HDR}, where a systematic framework is proposed to analyze these functional equations and their solutions. More broadly, the present work is part of a growing body of research exploring the deep connections between stochastic processes and integrable functional equations.

A major source of inspiration comes from the extensive literature on random walks in the quarter plane, where kernel methods have reached a high level of maturity. We refer in particular to the monograph \cite{fayolle_random_2017}, which provides a comprehensive account of these techniques. In parallel, complex-analytic approaches originating in queueing theory and stochastic networks \cite{foddy_analysis_1984,baccelli_analysis_1987} have played a fundamental role in the development of explicit formulas and boundary value problem methods.

Beyond providing a concise survey of existing results in the literature, one of the contributions of this paper is the derivation of a new functional equation satisfied by the transition probabilities of reflected Brownian motion. This equation can be seen as a unifying generalization of the classical functional equations associated with the invariant measure (in the recurrent case) and the Green function (in the transient case). 

This new equation can be analyzed and solved using the same integrable techniques as in the stationary and transient settings. To the best of our knowledge, such an approach has not previously been carried out in the literature. This opens the door to a unified treatment of a wide range of quantities associated with reflected Brownian motion.

\medskip

The paper is organized as follows. 
In Section~\ref{sec:SRBM}, we introduce the model of semimartingale reflected Brownian motion in the quadrant, together with its geometric interpretation in cones, and we recall the main recurrence and transience criteria. 
Section~\ref{sec:absesc} is devoted to absorption and escape probabilities, including hitting of the vertex and escape to infinity or along the axes, which provide fundamental examples of harmonic functions. 

In Section~\ref{sec:BAR}, we present the Basic Adjoint Relationships (BAR) for the transition density, the invariant measure, and the Green function, which serve as weak formulations of the underlying partial differential equations. 
These PDE aspects are further developed in Section~\ref{s1.3}, where we recall the associated potential theory, including oblique Neumann boundary conditions and Martin boundary theory.

In Section~\ref{sec:funceq}, we derive the kernel functional equations satisfied by the Laplace transforms of the transition density, invariant measure, Green functions, and harmonic functions, which form the starting point of the analytic approach. 
Section~\ref{sec:contRiem} is then devoted to the analytic continuation of these transforms on the Riemann surface associated with the kernel, leading to difference equations on the universal covering.

In Section~\ref{sec:intrep}, we solve these equations using Carleman boundary value problem techniques, leading to explicit integral representations. 
Section~\ref{sec:explifor} presents an alternative approach based on Tutte’s invariants, yielding explicit integral-free formulas in special parameter regimes. 
These results are complemented in Section~\ref{sec:diffgalois} by a classification of the Laplace transforms according to their algebraic and differential nature, using difference Galois theory.

Finally, Section~\ref{sec:asymptmartin} is devoted to asymptotic analysis via saddle-point methods and singularity analysis, while Section~\ref{sec:compa} presents the compensation approach in the degenerate setting, leading to explicit series representations of the invariant measure.

\medskip

Our results highlight the robustness of the kernel method and its versatility in stochastic analysis, reinforcing the deep connections between probability theory, complex analysis, and algebraic combinatorics.

\paragraph{A brief historical overview of reflected Brownian motion.}
The theory of reflected Brownian motion (RBM) lies at the crossroads of probability theory, partial differential equations, and applications such as queueing theory. Its development in queueing theory began in 1961 with Kingman’s seminal work, which introduced RBM as a powerful approximation for queueing systems operating under heavy traffic conditions \cite{Kingman1961}. In the 1980s, this line of research expanded considerably 
\cite{harrison_diffusion_1978,
harrison_brownian_1987,reiman_84_open}, with major contributions by Harrison, Reiman, Varadhan, and Williams, who developed multidimensional RBM models, particularly in orthants 
\cite{HaRe-81,
HaRe-81b,harrison_multidimensional_1987,reiman_boundary_1988}. Special attention was also given to two-dimensional systems, especially in wedges 
\cite{Williams-85,
williams_recurrence_1985,varadhan_brownian_1985}. 

From a probabilistic perspective, the origins of reflected diffusions trace back to the work of Skorokhod \cite{Skorokhod1961} on stochastic differential equations with reflection. This was followed by contributions of Watanabe \cite{watanabe1971} and the foundational work of Stroock and Varadhan \cite{Stroock1971}, who introduced the submartingale problem as a general framework for reflected diffusion processes. However, RBM in orthants falls outside the scope of these classical theories, due to the lack of boundary regularity. This leads to significant difficulties, including the treatment of nonsmooth boundaries and oblique reflection at corners, where standard analytical tools no longer directly apply.

From an analytical viewpoint, the well-posedness of reflected diffusions in smooth domains was established by \cite{LionsSznitman1984}, and later extended to oblique reflection and nonsmooth settings by \cite{DupuisIshii1993}. 
Approaches based on submartingale problems, such as \cite{Ramanan2006}, provide a robust framework for dealing with irregular domains. Recent works, including \cite{Burdzy2017}, focus on obliquely reflected Brownian motion in nonsmooth planar domains, combining probabilistic techniques with tools from complex analysis (such as conformal mappings) and excursion theory. 

Since the 1980s to the present day, RBM in cones and polyhedral domains has been extensively studied from multiple perspectives. A large body of work has been devoted to questions of existence, uniqueness and semimartingale properties 
\cite{varadhan_brownian_1985,
Williams-85, williams_semimartingale_1995, Burdzy2017}, as well as to phenomena such as non-pathwise uniqueness \cite{bass2024}. Recurrence and transience properties have also been investigated in depth in the 1980s \cite{williams_recurrence_1985} and 1990s
\cite{hobson_recurrence_1993, taylor_existence_1993, dupuis_lyapunov_1994, Ch-96} until quite recently in the 2010s
\cite{Br-11, BrDaHa-10, Kang2014}. Pathwise differentiability of reflected diffusions is also an important subject of study \cite{Lipshutz19}. 

Further developments concern the structure of invariant measures, including product-form stationary distributions such as the early papers
\cite{harrison_multidimensional_1987, williams_reflected_1987} and more recently
\cite{dieker_reflected_2009, OCOr-14}, explicit expressions \cite{franceschi_raschel_integral_2019,franceschi_bousquet_melou_price_hardouin_raschel_stationary_2023}, and asymptotic behavior
\cite{harrison_reflected_2009, dai_reflecting_2011, franceschi_kourkova_asymptotic_2017}. 
Lyapunov techniques have also played a central role in stability analysis 
\cite{dupuis_lyapunov_1994, sarantsev_reflected_2017}. In parallel, connections have been established with other stochastic processes 
\cite{biane_94, dubedat_reflected_2004, lega-87, Lep}, moving boundary problems \cite{Burdzy2009}. 
More recently, links with interacting particle systems, such as the Atlas model and competing particle systems, have further broadened the scope of RBM 
\cite{BanerjeeBudhiraja,ichiba_karatzas_degenerate_22}.

\section{SRBM in the quadrant}
\label{sec:SRBM}

\subsection{Definition, existence and uniqueness}

A semimartingale reflected Brownian motion (SRBM) in a quadrant is a continuous Markov process that evolves inside the quadrant, behaves as a Brownian motion with drift in its interior, is instantaneously reflected on each boundary edge along constant reflection directions, and spends zero Lebesgue time at the vertex.
\begin{deffr}[SRBM]
\label{intro:def:MBsemimartingale}
The process $Z_t=(Z_t^{(1)},Z_t^{(2)})^\top$ is called a semimartingale reflected Brownian motion in $\Rp^2$, driven by a Brownian motion $B_t$ with covariance matrix $\Sigma=(\sigma_{ij})_{1 \le i,j \le 2}\in\mathbb{R}^{2\times 2}$, drift $\mu=(\mu_i)_{i=1,2}\in\mathbb{R}^2$, and obliquely reflected on the axes according to a reflection matrix $R=(R_1,R_2)=(r_{ij})_{1 \le i,j \le 2}\in\mathbb{R}^{2\times 2}$ (where $R_j$ is its $j$-th column), if it admits the decomposition
\begin{equation}\label{eq:srbm}
Z_t = Z_0 + B_t + \mu t + R L_t, \qquad Z_t \in \Rp^2,
\end{equation}
where $L_t=(L_t^{(1)},L_t^{(2)})^\top$ is a continuous process with nondecreasing components, and each component $L_t^{(i)}$ increases only when the corresponding coordinate $Z_t^{(i)}$ is equal to $0$. The pair $(Z_t,L_t)$ is said to be a solution to the \textit{Skorokhod problem} in the quadrant for $B_t+\mu t$ with respect to the reflection matrix $R$. 
\end{deffr}

\paragraph{Normalization of parameters}
Without loss of generality, we will assume, by a simple rescaling argument, see Appendix~\ref{rescaling}, that 
\[ r_{11} = r_{22} = 1, \quad \sigma_{11}=\sigma_{22}=1 \quad \text{and} \quad \|\mu\|_2=1.\]
In this case \( L_t^{(1)} \) and \( L_t^{(2)} \) coincide with the boundary local times on the two edges of the quadrant.
For the sake of simplicity, we therefore adopt the following notation
\[
\Sigma :=
\begin{pmatrix}
1 & \rho\\
\rho & 1
\end{pmatrix} \text{ where } \rho\in(-1,1),
\qquad
R :=
\begin{pmatrix}
1 & r_2 \\
r_1 & 1
\end{pmatrix},
\qquad
\mu := \begin{pmatrix}
\mu_1 \\
\mu_2
\end{pmatrix} \text{ where } \mu_1^2+\mu_2^2=1.
\]
The existence and uniqueness conditions are now well understood \cite{reiman_boundary_1988,taylor_existence_1993}.
\begin{prop}[Existence and uniqueness]
SRBM in the quadrant exists for any initial distribution of $Z_0$ if and only if
\begin{equation}
\label{intro:existence}
(
1 - r_1 r_2 > 0 ) 
\quad \text{or} \quad 
(
r_1>0,\; r_2>0).
\end{equation}
In this case, the solution is unique in distribution for any given initial distribution.
\end{prop}
This condition is equivalent to the existence of a convex combination of the reflection vectors that points into the interior of the quadrant. 
Intuitively, it ensures that the corner is not too attractive, so that the process does not become trapped there.

\paragraph{Degenerate case}
It is worth noting that, as shown in \cite{ichiba_karatzas_degenerate_22}, the process
still exists when the covariance matrix is degenerate, that is,
when $\det \Sigma = 0$, i.e. $\rho=\pm 1$. In this situation, the driving Brownian motion
reduces to a one-dimensional Brownian motion, but the reflected process is a two-dimensional process.

\subsection{From the quadrant to convex cones}
\label{intro:subsec:cone}

Studying reflected Brownian motion in the quadrant is equivalent to studying it in cones. 
It is therefore possible to generalize to cones all results obtained in the quadrant. 
Indeed, a suitable linear transformation 
$$
T=
\begin{pmatrix}
\frac{1}{\sqrt{1-\rho^2}} & -\frac{\rho}{\sqrt{1-\rho^2}}\\
0 & 1
\end{pmatrix}
$$
maps a Brownian motion in the quadrant 
(with arbitrary non-degenerate covariance matrix $\Sigma$) to a Brownian motion in a cone of angle
\begin{equation}
\beta = 
\arccos\left( 
- \rho
\right)
\in (0,\pi)
,
\label{eq:beta}
\end{equation}
with identity covariance matrix and drift $\widetilde \mu:=T\mu$. The reflection angles $\epsilon$ and $\delta$, and the 
drift angle $\vartheta$, are then given by
\[
\tan \delta=\frac{\sin \beta}{r_{2}+\cos \beta}, 
\qquad 
\tan \epsilon=\frac{\sin \beta}{r_{1}+\cos \beta}, 
\qquad 
\tan \vartheta=\frac{\sin \beta}{\mu_1 / \mu_2+\cos \beta},
\]
see Figure~\ref{fig:linear_transformation1}.

\begin{figure}[hbtp]
 \vspace{-5mm}
 \centering
\includegraphics[scale=0.65]{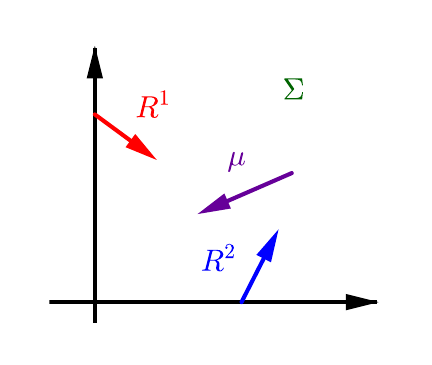} \hskip 10mm
\includegraphics[scale=0.65]{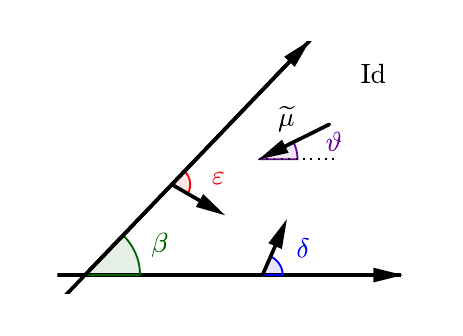}
 \vspace{-5mm}
 \caption{A linear transformation $T$ maps the quadrant to a cone of angle $\beta$; the new reflection angles are denoted $\epsilon$ and $\delta$, and the drift angle $\vartheta$}
 \label{fig:linear_transformation1}
\end{figure}

\paragraph{A key parameter}
We set
\begin{equation}
\label{intro:def:alpha}
\alpha=\frac{\delta+\epsilon-\pi}{\beta}.
\end{equation}
The value of this ratio $\alpha$ appears in many criteria and properties satisfied by the process.
For example, we find that the existence condition \eqref{intro:existence} of reflected Brownian motion as a semimartingale is equivalent to $\alpha<1$ \cite{Williams-85}.  
Also, $\alpha=0$ is the so-called \textit{Skew-symmetric case}, and the fact that $\alpha$ is a negative integer appears in the criterion of \citet{dieker_reflected_2009} for the stationary distribution to be a sum of exponentials.  
In the article \citeS{franceschi_bousquet_melou_price_hardouin_raschel_stationary_2023}, we show that the Laplace transform of the stationary distribution is differentially algebraic (see Section~\ref{sec:algdiffhier} for the definition) if $\alpha \in \mathbb{Z}+\frac{\pi}{\beta}\mathbb{Z}$.
Moreover, one can also show that, without drift, the process hits the origin almost surely if and only if $\alpha>0$.

Contrary to our current approach, the process could be defined not as a semimartingale (see Definition \ref{intro:def:MBsemimartingale}) but via a submartingale problem.  
See the article by \citet{varadhan_brownian_1985} for a rigorous study of such a process \textit{without drift} defined this way.  
If a solution to this problem exists, it is unique and satisfies the strong Markov property.  
A process satisfying the submartingale problem exists if and only if $\alpha< 2$.
When $\alpha \geqslant 2$, it is still possible to define a process absorbed or killed at $0$.


\subsection{Recurrence and transience conditions}
\label{intro:subsec:recurrence}

A Markov process exhibits two possible long-term behaviors: it is either
\emph{transient}, meaning that it escapes to infinity, or \emph{recurrent},
meaning that it returns infinitely often to every neighborhood of its starting
point. For a Feller strong Markov process $X$ on a locally compact state space $E$,
each state is either recurrent or transient, and communicating states share
the same classification. 
Transience is characterized by the fact that the process tends to infinity
almost surely, or equivalently that the Green measure of every compact set $K$ is
finite:
\[
G(x,K)=\mathbb{E}_x\!\left[\int_0^\infty \mathbf{1}_K(X_t)\,dt\right]<\infty.
\]
This quantity admits a natural interpretation: $G(x,K)$ is the expected total amount of time that the process, started from $x$, spends in the set $K$ over its entire lifetime.
Otherwise, when the process is not transient, the process is recurrent. In that case, there exists an invariant
measure (finite or $\sigma$-finite). A measure $\pi$ on $E$ is said to be
invariant if for all bounded measurable functions $f$ and all $t \ge 0$,
\[
\int_E \mathbb{E}_x[f(X_t)]\,\pi(dx)
=
\int_E f(x)\,\pi(dx).
\] 
This can be written in operator form as $\pi P_t=\pi$, where $P_t f(x)=\mathbb{E}_x[f(X_t)]$.
When the process is ergodic and $\pi$ is a probability measure, it admits the following interpretation: for any measurable set $A\subset E$,
\[
\pi(A) = \lim_{T\to\infty}\frac{1}{T}\int_0^T \mathbf{1}_{\{X_t\in A\}}\,dt
\quad \text{a.s.}
\]
In other words, $\pi(A)$ represents the 
proportion of time that the process spends in the set $A$.
The recurrent process is called \emph{positive recurrent} if the invariant
measure is finite, in which case it can be normalized into a stationary
probability distribution. Equivalently, expected return times to measurable
sets of positive measure are finite. 
It is called \emph{null recurrent} if the invariant measure is infinite
(but $\sigma$-finite), equivalently if expected return times to such sets
are infinite.

For SRBM in a wedge, the state space forms a single communicating class, so
the process is either entirely recurrent or entirely transient.
The following recurrence condition is due to Hobson and Rogers~\cite{hobson_recurrence_1993}, see also \cite{harrison_brownian_1987,
harrison_reflected_2009,franceschi_green_2021}. 

\begin{prop}[Recurrence and transience criterion]
Assume that the existence condition~\eqref{intro:existence} is satisfied and that the drift $\mu$ is non-zero.
Denote by $\mu_1^{-}$ and $\mu_2^{-}$ the negative parts of the drift components (defined by $\mu_i^-=\max(-\mu_i,0)$). 
The process $Z$ is transient if and only if
\begin{equation}
\mu_1 + r_{2}\mu_2^{-} > 0 
\quad \text{or} \quad 
\mu_2 + r_{1}\mu_1^{-} > 0,
\label{u11}
\end{equation}
and recurrent if and only if
\begin{equation}
\mu_1 + r_{2}\mu_2^{-} \le 0 
\quad \text{and} \quad 
\mu_2 + r_{1}\mu_1^{-} \le 0.
\label{v11}
\end{equation}
In the latter case, the process is positive recurrent and admits a unique stationary
distribution if and only if
\begin{equation}
\mu_1 + r_{2}\mu_2^{-} < 0 
\quad \text{and} \quad 
\mu_2 + r_{1}\mu_1^{-} < 0,
\label{v11pos}
\end{equation}
and is null recurrent if and only if
\[
\mu_1 + r_{2}\mu_2^{-} = 0 
\quad \text{or} \quad 
\mu_2 + r_{1}\mu_1^{-} = 0.
\]
In this case, the invariant measure is absolutely continuous with respect to the Lebesgue measure and we denote by $\pi$ its density.
\end{prop}

The conditions reflect a competition between the drift and the reflection vectors.
The quantities
\[
\mu_1 + r_{2}\mu_2^- 
\quad \text{and} \quad 
\mu_2 + r_{1}\mu_1^-
\]
admit a natural probabilistic interpretation. The term
$
\mu_1 + r_{2}\mu_2^-
$
represents the \emph{effective drift in the horizontal direction}, obtained by combining the interior drift $\mu_1$ with the average horizontal push generated by reflections on the boundary $\{x_2=0\}$. Indeed, when $\mu_2<0$, the process hits this boundary at a linear rate proportional to $\mu_2^-$, and each reflection contributes a horizontal component $r_{12}$. A completely symmetric interpretation holds for
$
\mu_2 + r_{1}\mu_1^-,
$
which corresponds to the \emph{effective vertical drift}.

The recurrence or transience of the process is therefore governed by the signs of these effective drifts: the process is transient as soon as one of them is strictly positive (providing an average escape direction), while it is recurrent when both are non-positive. Positive recurrence occurs when both are strictly negative, ensuring an overall inward drift, whereas null recurrence corresponds to the critical case where at least one effective drift vanishes.
These recurrence conditions are similar to those for reflected random walks in the quadrant \cite{fayolle_random_2017}.


Note that the recurrence conditions \eqref{u11} and \eqref{v11} do not involve the covariance matrix $\Sigma$. This is because we are in the non-zero drift case, and the drift dominates at large times and distances in the global behavior of the diffusion. In the zero-drift case, the recurrence condition can be expressed thanks to $\alpha$ defined in equation \eqref{intro:def:alpha} which depends on the covariance and the reflection vectors and the stationary distribution is explicit \citet{williams_recurrence_1985}.
\begin{prop}[Invariant measure, driftless case]
When $\mu=0$, the process is recurrent and admits an invariant measure if and only if $$0 \leqslant \alpha<2.$$
If we denote $p(r,\theta)$ the density of the invariant measure in polar coordinates of the corresponding reflected Brownian motion in a cone of angle $\beta$ represented in Figure~\ref{fig:linear_transformation1}, then 
we have $$p(r,\theta)=r^{-\alpha}\sin (\delta-\alpha\theta).$$
\end{prop}
Also, an explicit general expression for the stationary distribution of obliquely reflected Brownian motion \emph{without drift} in a planar domain was obtained in terms of a conformal map in \cite{harrison_landau_shepp_1985}.
When the drift is not zero, determining this invariant measure is, as we shall see, much more complicated.

\section{Absorption and escape probabilities}
\label{sec:absesc}

\subsection{Escape to infinity and absorption at the vertex}

We define the stopping time
$$
T:=\inf \{t>0: Z_t=0\} .
$$
Even when the existence condition~\eqref{intro:existence} is not satisfied, i.e. when $\alpha \geqslant 1$, the process $Z_t$ is well-defined until it hits the vertex at time $T$, which amounts to considering the process $\left({Z}_t\right)_{0 \leqslant t \leqslant T}$, see~\cite{taylor_existence_1993,lakner2023,franceschi_ernst_huang_escape_2021}. 

The probability of hitting the vertex,
$
\mathbb{P}(T<\infty)
$,
is sometimes called the ruin probability or the absorption probability when the process is killed at the vertex. The probability of not hitting the vertex, $\mathbb{P}(T=\infty)$, is called the survival or escape probability, depending on the context. In the context of a three-particle system, the gap process is a Brownian motion reflected within a quadrant, and the fact that $T=\infty$ means that there are no triple collisions. 

The driftless case $(\mu=0)$ was treated by Varadhan and Williams \cite{varadhan_brownian_1985}.  
In this case, the absorption probability does not depend on the starting point.
\begin{prop}[Probability of hitting the origin, driftless case]
When the drift is zero, i.e., $\mu=0$, we have
$$
\mathbb{P}[T<\infty]= \begin{cases}1 & \text { if } \alpha>0, \\ 0 & \text { if } \alpha \leqslant 0. \end{cases}
$$
If $\alpha \leqslant 0$, the vertex is not reached.  
If $0<\alpha<2$, the vertex is reached, but the amount of time spent at the vertex has zero Lebesgue measure.  
If $\alpha \geqslant 2$, the process hits the vertex and stays there.
\end{prop}

The case of non-zero drift poses a new challenge, as it is possible to have 
$$0<\mathbb{P}_{z}[T<\infty]<1$$
where this probability depends on the starting point $z$, see for example \citeS{franceschi_ernst_huang_escape_2021,franceschi_raschel_dual_2022,franceschi_flin_sumofexponential_2024}.
Define the hitting and non-hitting probabilities by
\[
h^0(z):=\mathbb{P}_z(T<\infty),
\qquad
h^\infty(z):=\mathbb{P}_z(T=\infty).
\]
and so
$$
h^0(z)+h^\infty(z)=1.
$$
When $\alpha \geq 1$, there is an almost sure equivalence between infinite lifetime and escape to infinity, see \citeS{franceschi_ernst_huang_escape_2021}. We write
$
Z_t \to \infty
$
to mean that
$
\|Z_t\|_2 \to \infty
$
as $t\to\infty$.
\begin{prop}[Hitting versus escape to infinity and boundary behavior at $0$ and at infinity]
Assume that $\alpha \geq 1$ and $\mu\neq 0$.
The events $\{T=\infty\}$ and $\{Z_t  \to \infty\}$ coincide almost surely, i.e.,
\[
\mathbb{P}\big(T=\infty,\; Z_t \not\to \infty\big)=0
\quad \text{and} \quad
\mathbb{P}\big(T<\infty,\; Z_t \to \infty\big)=0.
\]
Moreover,
\[
h^0(0)=1-h^\infty(0)=1,
\qquad
\lim_{z\to\infty} h^0(z)=1-\lim_{z\to\infty}h^\infty(z)=0.
\]
\end{prop}
In other words, the previous equation shows that the probability of hitting $0$ when starting from $0$ is equal to $1$, whereas this probability tends to $0$ as the starting point goes to infinity.

Let $(P_t^\dagger)_{t\ge 0}$ be the killed semigroup,
\[
P_t^\dagger f(z):=\mathbb{E}_z\!\big[f(X_t)\,\mathbf{1}_{\{t<T\}}\big],
\]
defined for bounded measurable $f$. This semigroup is \emph{sub-Markov} since
$P_t^\dagger \mathbf{1}(z)=\mathbb{P}_z(t<T)\le 1$.
The absorption and escape probabilities are (super)harmonic for this semigroup.

\begin{prop}[Harmonicity for the killed semigroup]
For all $t\ge0$,
\[
h^\infty = P_t^\dagger h^\infty,
\qquad
h^0 \ge P_t^\dagger h^0,
\]
i.e. $h^\infty$ is harmonic and $h^0$ is superharmonic for the killed
(sub-Markov) semigroup. 
\end{prop}
\begin{proof}
By the strong Markov property at time $t$,
\[
h^\infty(z)=\mathbb{P}_z(T=\infty)
=\mathbb{E}_z\!\Big[\mathbf{1}_{\{t<T\}}\,
\mathbb{E}_z\!\big[\mathbf{1}_{\{T=\infty\}}\mid \mathcal{F}_t\big]\Big]
=
\mathbb{E}_z\!\big[h^\infty(Z_t)\,\mathbf{1}_{\{t<T\}}\big]=P_t^\dagger h^\infty(z).
\]
Using the strong Markov property and $\mathbb{P}_z(T\le t)\ge 0$,
we obtain
$$h^0(z)=\mathbb{P}_z(T\le t)+\mathbb{E}_z[h^0(Z_t)\mathbf{1}_{\{t<T\}}]\;\ge\;\mathbb{E}_z\!\big[h^0(Z_t)\,\mathbf{1}_{\{t<T\}}\big]
= P_t^\dagger h^0(z).$$
\end{proof}


\subsection{Escape along an axis}
\label{subsec:levyfuite}

We now assume that
$$
\mu_1<0,
\quad
\mu_2<0,
$$
and
$$
r_{22}\mu_1 - r_{12}\mu_2>0, \quad
r_{11}\mu_2-r_{21} \mu_1>0.
$$
The effective drifts along the axes are positive, see Figure~\ref{fig:domination}. We are thus in a transient case and the process escapes to infinity. More precisely, it escapes to infinity along one of the two axes. We then define the probability of escape along the horizontal and vertical axes as
$$
h^\rightarrow (z):= \mathbb{P}_{z} (Z_1(t) \to\infty)
\quad \text{and} \quad
h^\uparrow(z):= \mathbb{P}_{z} (Z_2(t) \to\infty) .
$$
These probabilities depend on the starting point $z$ and satisfy the following properties, see \citeS{franceschi_fomichov_ivanovs_2022}. 
\begin{figure}[b]
\centering
\includegraphics[trim= 0mm 5mm 0mm 8mm,clip,width=0.4\textwidth]{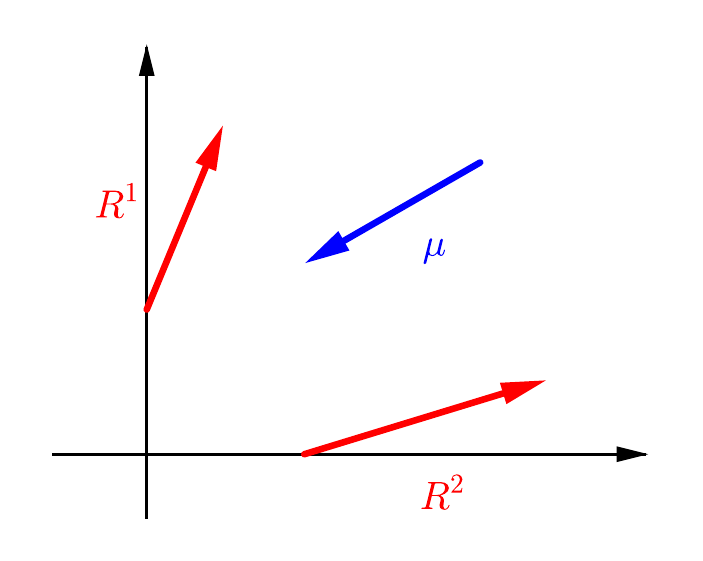}
\includegraphics[trim= 0mm 15mm 0mm 20mm,clip,width=0.5\textwidth]{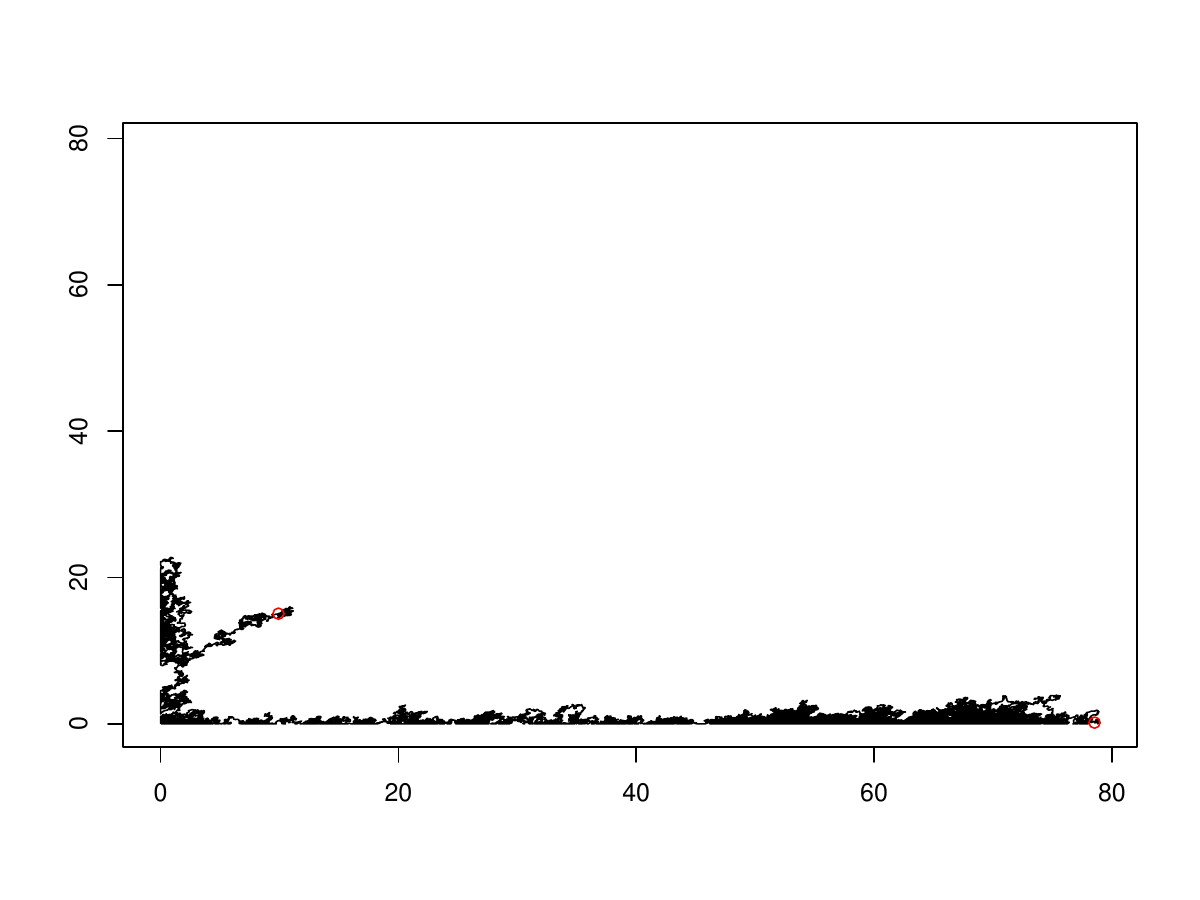}
\caption{Left: drift and reflection vectors; right: sample path of a reflected Brownian motion escaping along the horizontal axis.}
\label{fig:domination}
\end{figure}
\begin{prop}[Dichotomy of escape directions and boundary behavior]
For any starting point $z\in\mathbb{R}_+^2$, the events $\{Z_1(t)\to\infty\}$ and $\{Z_2(t)\to\infty\}$ form a partition up to a null set, i.e.,
\[
\mathbb{P}_z\big(Z_1(t)\to\infty,\; Z_2(t)\to\infty\big)=0,
\qquad
\mathbb{P}_z\big(Z_1(t)\not\to\infty,\; Z_2(t)\not\to\infty\big)=0.
\]
In particular, we have $h^\rightarrow(z)$ and $h^\uparrow(z)\in (0,1)$ and
$$h^\rightarrow(z)+h^\uparrow(z) =1.
$$
Moreover, for $z=(u,v)\in\mathbb{R}_+^2$,
 \begin{equation}
 \lim_{u\to\infty} h^\rightarrow(u,v)=1,
 \quad
 \lim_{v\to\infty} h^\rightarrow(u,v)=0.
\label{eq:limitvaluepra}
 \end{equation}
\end{prop}
In other words, the previous equation shows that when the starting point goes to infinity along the horizontal direction, the process escapes to infinity along this direction with probability tending to $1$, whereas if it goes to infinity along the vertical direction, this probability tends to $0$.

For a Markov process $(Z_t)_{t\ge 0}$, we denote by $\mathcal{T}$ the tail (asymptotic) sigma-field:
   $$\mathcal{T} := \bigcap_{t\ge 0} \sigma(Z_s : s\ge t).$$
Note that $A_1:=\{Z_1(t)\to\infty\}$ and $A_2:=\{Z_2(t)\to\infty\}\in \mathcal{T} $.
Let $(P_t)_{t\ge0}$ denote the Markov semigroup defined, for every bounded measurable function $f$, by
\[
P_t f(z):=\mathbb{E}_z[f(Z_t)].
\]
\begin{prop}[Tail events yield harmonic functions]
Let $A\in\mathcal{T}$ and define $h(z):=\mathbb{P}_z(A)$. Then, \[
{\quad h = P_t h \ \text{for all } t\ge 0, \quad}
\]
i.e.\ $h$ is harmonic for the Markov semigroup. In particular, $h^\rightarrow$ and $h^\uparrow$ are harmonic for the semigroup of $Z$.
\end{prop}
\begin{proof}
Let $(\mathcal{F}_t)_{t\ge 0}$ be the natural filtration, i.e. $\mathcal{F}_t=\sigma(Z_s:\, s\le t)$.
Since $A\in\mathcal{T}$, we have $A\in \sigma(Z_s:\, s\ge t)$, hence by the Markov property,
\[
\mathbb{P}_z(A\mid \mathcal{F}_t)=\mathbb{P}_{Z_t}(A)=h(Z_t)\qquad\text{a.s.}
\]
Taking expectations gives
\[
h(z)=\mathbb{P}_z(A)=\mathbb{E}_z\!\big[\mathbf{1}_A\big]
=\mathbb{E}_z\!\Big[\mathbb{P}_z(A\mid \mathcal{F}_t)\Big]
=\mathbb{E}_z\!\big[h(Z_t)\big]
= (P_t h)(z).
\]
\end{proof}
The study of positive harmonic functions using Martin’s boundary thus enables us to describe the asymptotic sigma-field $\mathcal{T}$.

\section{Basic Adjoint Relationships}
\label{sec:BAR}

In this section, we introduce the so-called \emph{Basic Adjoint Relationships} (BAR), which are integral identities characterizing the main quantities under study. These include the transition probabilities, the invariant measure in the recurrent case, and the Green functions in the transient case. The BAR can be viewed as weak formulations of the partial differential equations satisfied by these objects, which will be presented in the next section. These PDEs are adjoint to those satisfied by the harmonic functions introduced earlier.

The \emph{Basic Adjoint Relationships} lie at the heart of the derivation of the functional equation that will serve as the starting point for the analytical, algebraic and combinatorial methods presented in this article.
 
\subsection{Transition probability}

Let $Z$ be a Brownian motion with drift reflected at the boundary of the quadrant $\mathbb{R}_+^2$ associated with $(\Sigma, \mu, R)$ such that
$$ Z_t=z_0 + W_t + \mu t + R L_t \ \in \mathbb{R}_+^2.$$ 
We denote the transition probability of the process by $\mathbb{P}_{z_0}(Z_t\in dz)=P(t,z_0,dz)=p(t,z_0,z) dz$.
For any bounded measurable function $f$,
\[
\mathbb{E}_{z_0}\!\left[\int_0^t f(s,Z_s)\,ds\right]
=\int_0^t\!\int_{\mathbb{R}_+^2} f(s,z)\,p(s,z_0,z)\,ds\,dz.
\]
We draw inspiration from this formula to define two new measures.
One may introduce boundary measures $p_1^{z_0}$ and $p_2^{z_0}$ such that for any bounded measurable function $g$, 
\[
\mathbb{E}_{z_0}\!\left[\int_0^t f(s,Z_s)\,dL_s^1\right]
=\int_0^t\!\int_{\mathbb{R}_+^2} f(s,z)\,p_1^{z_0}(s,z)\,ds\,dz,
\]
\[
\mathbb{E}_{z_0}\!\left[\int_0^t f( s,Z_s)\,dL_s^2\right]
=\int_0^t\!\int_{\mathbb{R}_+^2} f(s,z)\,p_2^{z_0}(s,z)\,ds\,dz.
\]
Since $dL_t^1$ and $dL_t^2$ charge only the boundary faces $
F_1=\{(0,z_2):z_2\ge0\}$ and
$F_2=\{(z_1,0):z_1\ge0\}$, $p_1^{z_0}$ has its support on $\mathbb{R}\times F_1$ and $p_2^{z_0}$ on $\mathbb{R}\times F_2$.

In fact, it can be shown that, up to a multiplicative constant, the boundary measures coincide with the boundary values of the (continuous) transition density. More precisely, for $z=(z_1,z_2)$,
\[
2 p_1^{z_0}(t,z)\propto p\bigl(t,z_0,(0,z_2)\bigr) \delta_0(z_1),
\qquad
2 p_2^{z_0}(t,z)\propto p\bigl(t,z_0,(z_1,0)\bigr) \delta_0(z_2) 
\]
where $\delta_0$ is the Dirac delta function. Without the normalization of the parameters, the factor $2$ must be replaced by $2r_{11}/\sigma_{11}$.

Let $\mathcal{G}$ be the infinitesimal generator of $(W_t +\mu t)$, which is also the generator of $Z$ in the interior of $\mathbb{R}_+^2$. We have
\begin{equation}
\label{eq:generateur}
\mathcal{G} f(z)
=
\frac{1}{2}\left(
\frac{\partial^2 f}{\partial z_1^2}(z)
+ 2\rho \frac{\partial^2 f}{\partial z_1 \partial z_2}(z)
+ \frac{\partial^2 f}{\partial z_2^2}(z)
\right)
+ \mu_1 \frac{\partial f}{\partial z_1}(z)
+ \mu_2 \frac{\partial f}{\partial z_2}(z).
\end{equation}

\begin{prop}[Basic Adjoint Relationship for the transition density]\label{barpt}
For any $f\in C^{1,2}([0,t]\times\mathbb{R}_+^2)$, such that $f$ is non-negative, bounded, as are its first and second-order derivatives, and $f(t,z)\to 0$ when $t\to \infty$ uniformly in $z$,
\begin{equation}
\begin{aligned}\label{eq:barpt}
0=f(0,z_0)
+\int_0^\infty\!\int_{\mathbb{R}_+^2}(\partial_t f+\mathcal G f)(t,z)\,p(t,z_0,z)\,dz\,dt+
\sum_{i=1}^2 \int_0^\infty\!\int_{\mathbb{R}_+^2} (R^i\!\cdot\nabla f)(t,z)\,p_i^{z_0}(t,z) \,dz\,dt.
\end{aligned}
\end{equation}
\end{prop}
\begin{proof}
Applying Itô's formula to $f(t,Z_t)$ yields
\[
\begin{aligned}
f(t,Z_t)=f(0,z_0)
+\int_0^t (\partial_s+\mathcal G)f(s,Z_s)\,ds
+\int_0^t \nabla f(s,Z_s)\cdot dW_s 
+\sum_{i=1}^2\int_0^t (R^i\cdot\nabla f)(s,Z_s)\,dL_s^i .
\end{aligned}
\]
Taking expectations eliminates the martingale term and gives
\[
\begin{aligned}
\mathbb{E}_{z_0}[f(t,Z_t)]
=f(0,z_0)
+\int_0^t \mathbb{E}_{z_0}[(\partial_s+\mathcal G)f(s,Z_s)]\,ds
+\sum_{i=1}^2
\mathbb{E}_{z_0}\!\left[\int_0^t (R^i\!\cdot\nabla f)(s,Z_s)\,dL_s^i\right].
\end{aligned}
\]
Using the transition density $p_t(z_0,z)$ we write
\[
\mathbb{E}_{z_0}[(\partial_s+\mathcal G)f(s,Z_s)]
=\int_{\mathbb{R}_+^2}(\partial_s+\mathcal G)f(s,z)p_s(z_0,z)\,dz .
\]
For the boundary terms, since $dL_s^i$ charges only $F_i$, we obtain
\[
\mathbb{E}_{z_0}\!\left[\int_0^t (R^i\!\cdot\nabla f)(s,Z_s)\,dL_s^i\right]
=\int_0^t\!\int_{\mathbb{R}_+^2} (R^i\!\cdot\nabla f)(s,z)\,q_i(ds,dz).
\]
Putting everything together yields, for any $t>0$,
\[
\mathbb{E}_{z_0}[f(t,Z_t)]
=f(0,z_0)
+\int_0^t\!\int_{\mathbb{R}_+^2}(\partial_s+\mathcal G)f(s,z)p_s(z_0,z)\,dz\,ds 
+\sum_{i=1}^2 \int_0^t\!\int_{\mathbb{R}_+^2} (R^i\!\cdot\nabla f)(s,z)\,q_i(ds,dz).
\]
We now let $t\to\infty$. By assumption, $f(t,\cdot)\to 0$ uniformly and $f$ is bounded, hence
\[
\mathbb{E}_{z_0}[f(t,Z_t)] \longrightarrow 0.
\]
Moreover, by dominated convergence,
\[
\int_0^t \!\int_{\mathbb{R}_+^2}(\partial_s+\mathcal G)f(s,z)p_s(z_0,z)\,dz\,ds
\longrightarrow
\int_0^\infty \!\int_{\mathbb{R}_+^2}(\partial_s+\mathcal G)f(s,z)p_s(z_0,z)\,dz\,ds,
\]
and similarly for the boundary terms. Passing to the limit gives exactly \eqref{eq:barpt}.
\end{proof}
 
\subsection{Invariant measure}
\label{intro:subsec:mesinvgreen}

We assume that conditions \eqref{v11pos} are satisfied so that the process is positive recurrent and admits a unique stationary distribution (or invariant measure) $\pi$. This measure satisfies
$$
\pi(A)=\mathbb{E}_{\pi} \left[  \int_0^1 1_{\{Z_s \in A\}} \mathrm{d}s \right]= \mathbb{E}_{\pi} \left[ \frac{1}{t} \int_0^t 1_{\{Z_s \in A\}} \mathrm{d}s \right]
\text{ for all } t>0.
$$
We define $\nu_1$ and $\nu_2$ as measures on the boundaries,
such that for any $A\subset\mathbb{R}_+^2$ and for $i=1,2$,
\begin{equation}
\label{intro:eq:def:mesinvbords}
\pi_i (A) := \mathbb{E}_{\pi} \left[ \int_0^1 1_{\{Z_s \in A\}} \mathrm{d}L_s^i \right]
= \mathbb{E}_{\pi} \left[ \frac{1}{t} \int_0^t 1_{\{Z_s \in A\}} \mathrm{d}L_s^i \right]
\text{ for all } t>0
.
\end{equation}
The support of $\pi_1$ lies on the face $F_1$ and that of $\pi_2$ on the face $F_2$. All these measures admit densities with respect to the Lebesgue measure which we also denote by $\pi$, $\pi_1$ and $\pi_2$, with a slight abuse of notation.
We may thus say that $\pi_i (A)$ measures the \textit{proportion of local time} that the process spends on the boundary at $A$. These are known as boundary invariant measures.
By the definition of $\pi_i$, one can show that
\begin{equation}
\label{intro:eq:invbordformule}
\mathbb{E}_{\pi} \left[ \frac{1}{t} \int_0^t f(Z_s) \mathrm{d}L_s^i  \right]= \int_{\mathbb{R}_+^2} f(x) \pi_i (\mathrm{d} x), \quad \text{for } f\geqslant 0.
\end{equation}
As was the case for the transition density, the boundary invariant measures are equal to the invariant density on the boundary, up to a multiplicative constant. More precisely, for $z=(z_1,z_2)$,
\[
2 \pi_1(z_1,z_2)\propto \pi\bigl(0,z_2) \delta_0(z_1),
\qquad
2 \pi_2(z_1,z_2)\propto \pi(z_1,0) \delta_0(z_2) ,
\]
see \cite[Section 2.2]{franceschi_bousquet_melou_price_hardouin_raschel_stationary_2023}.
The following proposition is proven in a similar way to the BAR of Proposition~\ref{barpt}, see \cite{harrison_brownian_1987,dai_reflected_1992}.
\begin{prop}[Basic Adjoint Relationship for the stationary distribution]
\label{prop:bar}
For all $f \in \mathcal{C}^2_b (\mathbb{R}_+^2)$ ($f$ is $\mathcal{C}^2$ and its first and
second-order derivatives are bounded functions), we have
\begin{equation} \label{d} \int_{\mathbb{R}_+^2} \mathcal{G} f(z) \pi (\mathrm{d} z) + \sum_{i=1,2} \int_{\mathbb{R}_+^2} R_i \cdot \nabla f(z) \pi_i (\mathrm{d} z) =0 .
\end{equation}
\end{prop} 
This relation characterizes the stationary distribution: if $\pi$, $\pi_1$ and $\pi_2$ have positive, integrable densities satisfying \eqref{d}, then they are the invariant measures of the process.

\subsection{Green measure}
\label{subsec:green125}

We now assume that conditions \eqref{u11} are satisfied so that the process is transient. Until recently, the transient case had received relatively little attention in the literature. We can define Green functions (or measures) as follows.
We recall that for a starting point $z_0\in\mathbb{R}^2$ and a measurable set $A \subset \mathbb{R}_+^{2}$ the Green function is defined by
$$
G(z_0,A)= \int_0^{\infty} P(t,z_0,A)\mathrm{d} t = \mathbb{E}_{z_0} \left[ \int_{0}^{\infty} \mathbf{1}_{\{Z_t \in A \}} \mathrm{d} t \right] 
$$
where $P(t,{z_0},A)=\mathbb{P}_{z_0} [Z_t \in A ]=\mathbb{E}_{z_0} [\mathbf{1}_{Z_t \in A}]$. The density of $G(z_0, \cdot)$ is given by
$$g({z_0},z) =\int_0^{\infty} p(t,z_0,z)\mathrm{d} t,$$
where $p(t,z_0,z)$ is the transition density of the process.
The Green measure $G(z_0,A)$ thus measures the \textit{total expected time} spent in $A$ starting from $z_0$.
We define the Green operator classically as follows:
\begin{equation}
\label{eq:operateurG}
G f(z_0):=\mathbb{E}_{z_0}\left[\int_0^{\infty} f(Z_t) \mathrm{d} t\right]=\int_{\mathbb{R}_{+}^2} f(z) g(z_0, z) \mathrm{d} z .
\end{equation}
As in the recurrent case for invariant measures, we define $g_1^{z_0}$ and $g_2^{z_0}$ as the Green functions (or measures) on the boundary.
For $i=1,2$, we have
$$
g_i^{z_0} (A):= \mathbb{E}_{z_0} \left[ \int_0^\infty 1_{\{Z(s) \in A\}} \mathrm{d}L_s^i \right] <\infty .
$$
The support of $g_1^{z_0}$ is on $F_1$ and that of $g_2^{z_0}$ on $F_2$. There is also a relationship linking Green's measures on the edges to those in the interior
\[
2 g_1^{z_0}(z)\propto g\bigl(z_0,(0,z_2)\bigr) \delta_0(z_1),
\qquad
2 g_2^{z_0}(z)\propto g\bigl(z_0,(z_1,0)\bigr) \delta_0(z_2) .
\]
Thus, $g_i^{z_0} (A)$ measures the \textit{total expected local time} spent on a boundary in $A$ starting from $z_0$.
We obtain the analog of formula \eqref{intro:eq:invbordformule}, yielding the operator
\begin{equation}
\label{intro:eq:greenbordformule}
G_i f(z_0):= \mathbb{E}^{z_0} \left[\int_0^\infty f(Z_t) \mathrm{d}L^i_t \right] = \int_{\mathbb{R}_+} f(x) g_i^{z_0} (x) \mathrm{d}x.
\end{equation}
We then have a \emph{transient BAR} analogous to the recurrent case, which characterizes the Green's measure and follows easily from Itô's formula, see \citeS{franceschi_green_2021}.
\begin{prop}[Basic adjoint relationship, transient case]
\label{prop:bartransient} Assume that the process is transient.
For all $f \in \mathcal{C}^2_0 (\mathbb{R}_+^2)$ (i.e., $\mathcal{C}^2$ functions on $\mathbb{R}_+^2$ tending to $0$ at infinity), we have
\begin{equation}
0=f\left(z_0\right)+\int_{\mathbb{R}_{+}^2}\left(\frac{1}{2} \nabla \cdot \Sigma \nabla+\mu \cdot \nabla\right) f(z) g^{z_0}(z) \mathrm{d} z+\sum_{i=1}^2 \int_{\mathbb{R}_{+}} R_i \cdot \nabla f(z) g_i^{z_0}(z) \mathrm{d} z .
\label{eq:bartransient}
\end{equation}
\end{prop}

\section{Potential theory}
\label{s1.3}

This section briefly presents some classical results from potential theory and its probabilistic counterpart (see the reference book \cite{Doob2001}), applied to reflected stochastic processes in a quadrant. We first focus in Section~\ref{sec:neumann} on partial differential equations with Neumann boundary conditions, which arise when considering reflected processes. Then, in Section~\ref{sec:martharm}, we present Martin boundary theory, illustrating it using the processes of interest.

\subsection{Oblique Neumann boundary condition in a quadrant}
\label{sec:neumann}

Green functions and invariant measures of Markov processes play a central role in potential theory and ergodic theorems for additive functionals. In particular, they provide a probabilistic interpretation of the solutions of certain partial differential equations. Our case is particularly challenging, as we consider a non-smooth and unbounded domain, with oblique reflection at the boundary.

The generator of the SRBM in the interior of the quadrant $\mathcal{G}$ and its adjoint generator $\mathcal{G}^*$ are given by
$$
\mathcal{G} f=\frac{1}{2} \nabla \cdot \Sigma \nabla+\mu \cdot \nabla \quad \text { and } \quad \mathcal{G}^* f=\frac{1}{2} \nabla \cdot \Sigma \nabla-\mu \cdot \nabla .
$$
For simplicity, in this section only, we will replace the notation for the final point, $z$, with $y=(y_1,y_2)$, and the notation for the starting point, $z_0$, with $x=(x_1,x_2)$. Let $\mathcal{G}_x$ denote the operator associated with the variable $x$ and $\mathcal{G}^*_y$ the operator associated with the variable $y$. Harrison and Reiman \cite[(8.2) and (8.3)]{HaRe-81} informally derive the Kolmogorov backward and forward equations (with boundary and initial conditions) for $p(t,x, y)$, the transition density of the process.

\begin{prop}[Kolmogorov equations]
The backward equation can be written as
$$
\left\{\begin{array}{l}
\mathcal{G}_x p_t(x, y)=\partial_t p_t(x, y), \\
R_i \cdot \nabla_x p_t(x, y)=0 \text { if } x_i=0, \\
p_0(x, y)=\delta_0 (y-x) .
\end{array}\right.
$$
%
The forward equation (or Fokker–Planck equation) can be written as
$$ \hspace{2.1cm}
\left\{\begin{array}{l}
\mathcal{G}_y^* p_t(x, y)=\partial_t p_t(x, y), \\
R_i^* \cdot \nabla_y p_t(x, y)-2 \mu_i p_t(x, y)=0 \text { if } y_i=0, \\
p_0(x, y)=\delta_0 (y-x) ,
\end{array}\right.
$$
where
$$
R^*=2 \Sigma-R=
\begin{pmatrix}
1 & 2\rho-r_2
\\
2\rho-r_1 & 1
\end{pmatrix}
$$
and $R_j^*$ is its $j$-th column. 
\end{prop}
The weak form of the forward equation is exactly the Basic Adjoint Relationship for the transition density
\eqref{eq:barpt}.
Letting $t\to\infty$ in the forward equation, Harrison and Reiman conclude that, in the positive recurrent case, the density $\pi$ of the stationary distribution satisfies the steady-state equation \cite[(8.5)]{HaRe-81}
$$
\left\{\begin{array}{l}
\mathcal{G}^*_y \pi(y)=0, \\
R_i^* \cdot \nabla_y \pi(y)-2 \mu_i \pi=0 \text { if } y_i=0.
\end{array}\right.
$$
The weak form of the previous PDE is exactly the Basic Adjoint Relationship in the recurrent case \eqref{d}.
In the transient case, integrating the forward equation in time from $0$ to $\infty$ suggests that the Green function $g$ satisfies the following partial differential equation with
\emph{Robin boundary conditions} (i.e., prescribing a linear combination of the function and its derivative on the boundary):
$$
\left\{\begin{array}{l}
\mathcal{G}_y^* g(x, y)=-\delta_x(y), \\
R_i^* \cdot \nabla_y g(x, y)-2 \mu_i g(x, y)=0 \text { if } y_i=0 .
\end{array}\right.
$$
The weak form of the previous PDE is exactly the Basic Adjoint Relationship in the transient case
\eqref{eq:bartransient}.

\begin{rem}[Neumann boundary problem]
\label{rem:neumannprob}
The Green function $g$ of obliquely reflected Brownian motion in the quadrant is thus a fundamental solution of the adjoint operator $\mathcal{G}^*$. Along with the boundary Green functions $g_i$, it allows us to find solutions to the following \emph{oblique Neumann boundary problem}
$$
\begin{cases}\mathcal{G}_x u(x)=-f(x) & \text { in } \mathbb{R}_{+}^2, \\ R_i \cdot \nabla_x u(x)=\varphi_i(x) & \text { if } x_i=0. \end{cases}
$$
Recall that the definitions of the operators $G$ and $G_i$ are given in~\eqref{eq:operateurG} and~\eqref{intro:eq:greenbordformule}. If a solution exists, then
$$
u=G f+G_1 \varphi_1+G_2 \varphi_2 
$$
should satisfy the oblique Neumann boundary problem. Since the domain is unbounded, uniqueness is not guaranteed unless additional conditions such as $u(x)\to 0$ as $|x|\to \infty$ are imposed. In fact, the corresponding harmonic functions $h$ (i.e., functions solving the boundary problem with $f=0$ and $\varphi_i=0$) are not always unique. This occurs, for instance, when the Martin boundary is not reduced to a single point, as we shall see in the next section. Therefore, if $u$ is a solution of the Neumann boundary problem and $h$ is harmonic, then $u+h$ is also a solution.
\end{rem}

\subsection{Harmonic functions}
\label{sec:martharm}

The goal of this section is to recall the notion of harmonic function for a Markov process, as well as the main results from Martin boundary theory, and to illustrate them using a reflected process in the quadrant.

Let $X_t$ be a transient Markov process (e.g., reflected Brownian motion) on a state space $M$ (e.g., the quadrant). 
A function $h$ is harmonic in $M$ for the process $X$ 
if the mean value property
$$
\mathbb{E}_x\left[h\left(X_{\tau_K}\right)\right]=h(x)
$$
holds for every compact $K \subset M$, where $\tau_K$ is the first exit time from $K$.
The function $h$ is superharmonic for $X$ if $$\mathbb{E}_x\left[f\left(X_{\tau_K}\right)\right] \leq f(x)$$ for every compact $K$.
A nonnegative harmonic function $h$ is minimal if, for every harmonic function $g$ such that $0 \leq g \leq h$, we have $g = c h$ for some constant $c$.
The following result states that a function is harmonic if it satisfies the Neumann problem of Remark~\ref{rem:neumannprob} with $f=0$ and $\varphi_i=0$.
\begin{prop}[Harmonic function and PDE]
$\mathcal{C}^2$ harmonic functions for the reflected Brownian motion in the quadrant are those that vanish under the generator and the boundary generator, i.e., functions $h \in \mathcal{C}^2\left( \mathbb{R}_{+}^2\right)$ such that
\begin{equation}
\begin{cases}\mathcal{G} h(x)=0 & \text { in } \mathbb{R}_{+}^2, \\ R_i \cdot \nabla_x h(x)=0 & \text { if } x_i=0.\end{cases}
\label{eq:pdeharmneuman}
\end{equation}
\end{prop}
\begin{proof}[Sketch of proof]
Sufficient condition can be shown directly using Dynkin's formula. For the necessary condition, the interior equation follows directly, since it coincides with that of an unconstrained Brownian motion. The Neumann boundary condition can then be deduced from \cite[Corollary~3.3]{andres_2009}. 
\end{proof}

\paragraph{Escape and absorption probabilities}
It follows from Section~\ref{sec:absesc} that the escape probabilities along the axes, $h^\rightarrow(z)$ and $h^\uparrow(z)$, are harmonic for the reflected Brownian motion. Moreover, the escape probability to infinity, $h^\infty$, is harmonic for the reflected Brownian motion killed at the origin. All of these functions satisfy the PDE~\eqref{eq:pdeharmneuman}.

\subsection{Martin boundary}
We recall a few key results from Martin boundary theory. For more details, see these classic references \cite{chung2005,kunita1963,martin1941,Pinsky1995}. Recall that the Green function is denoted by $g(x,y)$.
For a fixed reference state $x_0$, the \emph{Martin kernel} is defined by
$$
k(x,y):=\frac{g(x,y)}{g(x_0,y)}.
$$
The \emph{Martin compactification} $\overline{M}$ is the smallest compactification of $M$ such that $y \mapsto k(x,y)$ extends continuously. The \emph{Martin boundary} is defined as
$$
\partial M:=\overline{M} \backslash M .
$$
The function $x \mapsto k(x,y)$ is superharmonic for every $y \in \overline{M}$. The \emph{minimal Martin boundary} is
$$
\partial_m M:=\left\{y \in \partial M: x \mapsto k(x,y) \text { is minimal harmonic}\right\} .
$$
Finally, we state one of the main results of Martin theory, which gives an integral representation of positive harmonic functions \cite[Theorem 4]{kunita1963}.
\begin{thm}[Martin representation theorem]
For any positive harmonic function $h$, there exists a unique finite measure $m$ on the minimal Martin boundary such that for all $x \in M$,
$$
h(x)=\int_{\partial_m M} k(x,y) m(\mathrm{d} y).
$$
\end{thm}

Martin’s boundary and the associated harmonic functions can thus be determined by obtaining the exact asymptotics of the Green’s functions in all directions. The method for obtaining this asymptotic result is explained in Section~\ref{sec:asymptmartin}.
The Martin boundaries and harmonic functions for reflected Brownian motion in the quadrant and the half-plane are determined in \citeS{franceschi_kourkova_petit_asymptotics_2024,franceschi_ernst_asymptotic_2021}, in the case of a degenerate process in~\cite{petit2024} and for space-time Brownian motion killed in a cone in~\citeS{franceschi_spacetime_2024}. 

\section{Functional equations}
\label{sec:funceq}

Functional equations involving Laplace transforms are the continuous analogue of the fundamental functional equation for random walks in the quarter plane, which involves generating functions. They derive from the BAR or from the PDEs and are the starting point of several analytic, combinatorial and algebraic approaches.

\subsection{General case: transition density}

Recall that we denote by $p(t,z_0,z)$ the transition density of the degenerate reflected Brownian motion in the quarter plane, with $z=(z_1,z_2)$. We define its Laplace transform in time and space by
\begin{equation}
\label{eq:Lap_main}
 \widehat p (u,x,y)= \mathbb{E}_{z_0} \!\left[ \int_0^\infty e^{x Z_t^{(1)}+y Z_t^{(2)} + ut} \, dt \right]
 = \int_{\mathbb R_+^2}\!\int_{0}^\infty e^{x z_1+yz_2+ut}p(t,z_0,z)\, dt\, dz.
\end{equation}
For the sake of brevity, we omit the dependence on $z_0$, even though it is clearly present.
We also consider the joint Laplace transform in space and time of the boundary measures:
\begin{equation}
\label{eq:Lap1_main}
    \widehat p_1(u,y) = \mathbb{E}_{z_0} \!\left[ \int_0^\infty e^{y Z_t^{(2)} + ut} \, dL_t^1 \right]
    = \int_{\mathbb R_+}\!\int_{0}^\infty e^{yz_2+ut} \, p_1^{z_0}(dt, dz_2),
\end{equation}
and similarly for $\widehat p_2(u,x)$.
These quantities can be interpreted as the resolvent applied to exponential test functions. 
To ensure convergence, we first assume that $\Re x<0$, $\Re y<0$, and $\Re u<0$. In fact, we will later show that these Laplace transforms admit analytic continuations to the complex plane cut along a half-line.
The following proposition states a functional equation that characterizes the transition density.

\begin{prop}[Functional equation for the transition density]
For $x,y,u\in\mathbb{C}$ such that $\Re x<0$, $\Re y<0$, and $\Re u<0$, we have
\begin{equation}
    \label{eq:main_eq}
    -\bigl(K(x,y)+u\bigr)\widehat{p}(u,x,y)
    = k_1(x,y) \widehat{p}_1(u,y)
    + k_2(x,y) \widehat{p}_2(u,x)
    + e^{\langle (x,y),z_0\rangle},
\end{equation}
where $K,k_1,k_2$ are defined by
\begin{equation}
\label{eq:def_parameters_Maxence}
    \left\{
    \begin{array}{rcl}
    K(x,y) & = & \frac{1}{2}(x^2+y^2+2\rho xy)+\mu_1x+\mu_2y,\\
    k_1(x,y) & = & x+r_1y,\\
    k_2(x,y) & = & r_2x+y.
    \end{array}
    \right.
\end{equation}
\end{prop}

\begin{proof}
This follows by applying the Basic Adjoint Relationship (BAR) for the transition density (see Proposition~\ref{barpt}) to the test function $f(t,z)=e^{xz_1+yz_2+ut}$, which satisfies the required assumptions for $t>0$, $z\in\mathbb{R}_+^2$, and $\Re x<0$, $\Re y<0$, and $\Re u<0$.
\end{proof}

The functional equation~\eqref{eq:main_eq} generalizes two equations that have been extensively studied: the functional equation for the invariant measure in the recurrent case (see~\eqref{eq:EJP-25} below) and the functional equation for the Green functions in the transient case (see~\eqref{eq:Petit-24} below). 
We explain below how to recover these two equations.
Equation~\eqref{eq:main_eq} not only encompasses these two important situations, but also provides a unified framework to address one of the fundamental questions for SRBM, namely the explicit computation of its distribution, or at least of its Laplace transform.

\subsection{Recurrent case: invariant measure}

Recall that we denote by $\pi(z)$ the density of the invariant (stationary) distribution, with $z=(z_1,z_2)$.
We introduce its Laplace transform in space by
\begin{equation}
\label{eq:Lap_pi}
\widehat \pi (x,y)
=
\int_{\mathbb R_+^2} e^{x z_1+yz_2}\pi(z)\,dz.
\end{equation}
We also define the Laplace transforms of the boundary stationary measures:
\begin{equation}
\label{eq:Lap_pi1}
\widehat \pi_1(y)
=\mathbb{E}_{\pi} \int_0^1 e^{y Z_t^{(2)}} dL_t^1=
\int_{\mathbb R_+} e^{y z_2}\,\nu_1(dz_2),
\end{equation}
and similarly for $\widehat \pi_2(x)$.
To ensure convergence, we assume that $\Re x<0$ and $\Re y<0$. As before, these transforms admit analytic continuations to suitable domains of the complex plane.
The following proposition states the functional equation satisfied by these transforms.

\begin{prop}[Functional equation for the invariant measure]
Assume that conditions \eqref{v11pos} are satisfied so that the process is positive recurrent. For $x,y\in\mathbb{C}$ such that $\Re x<0$, $\Re y<0$, we have
\begin{equation}
\label{eq:EJP-25}
- K(x,y)\widehat{\pi}(x,y)
=
k_1(x,y)\widehat{\pi}_1(y)
+
k_2(x,y)\widehat{\pi}_2(x),
\end{equation}
where $K,k_1,k_2$ are defined in~\eqref{eq:def_parameters_Maxence}.
\end{prop}

\begin{proof}
This follows by applying the stationary Basic Adjoint Relationship (BAR) (see Proposition~\ref{prop:bar}) to the test function $f(z)=e^{x z_1+y z_2}$, which satisfies the required assumptions for $z\in\mathbb{R}_+^2$ and $\Re x<0$, $\Re y<0$.
\end{proof}
\begin{rem}[From the functional equation for the transition density to the one for the invariant measure]
Without going into full details, we briefly explain how, at a formal level, the functional equation~\eqref{eq:EJP-25} can be recovered from the general equation~\eqref{eq:main_eq} by letting $u\to 0$. Under standard ergodicity assumptions, the transition density converges to the invariant density, i.e.,
\[
\pi(z)=\lim_{t\to\infty} p(t,z_0,z),
\]
and, with the final value theorem, one has
\[
\lim_{s\to 0} u \int_0^\infty e^{-ut} p(t,z_0,z)\,dt = \pi(z).
\]
This shows that $\widehat \pi (x,y)$ arises as the limit of $u\,\widehat p(u,x,y)$ as $u\to 0$. 
Similarly, one obtains the analogue of the above limit for the boundary measures.  
Ergodic properties suggests that, starting from $z_0$, the boundary occupation measure satisfies the convergence
\[
\frac{1}{t}\,\mathbb{E}_{z_0}\!\left[\int_0^t \mathbf{1}_{\{Z(s)\in A\}}\, dL_s^1 \right]
\;\longrightarrow\;
\pi_1(A),
\qquad t\to\infty,
\]
at least at a formal level.
We obtain the first equality below by applying the final value theorem and the last equality by applying the ergodic properties:
\begin{align*}
\label{eq:final_value_boundary}
\lim_{s\to 0}
u\,\widehat p_1(u,y)
&= \lim_{t\to\infty}\int_{\mathbb R_+^2} e^{y z_2 }p_1(t,z_2) d z_2
=\lim_{t\to\infty} \frac{1}{t} \int_0^t \int_{\mathbb{R}_+^2} e^{y z_2 }p_1(t',z_2) d z_2 dt'
\\ &=\lim_{t\to\infty} \frac{1}{t}
\mathbb{E}_{z_0} \!\left[ \int_0^t e^{y Z_t^{(2)}}\, dL_t^1 \right]
=\mathbb{E}_{\pi} \!\left[ \int_0^1 e^{y Z_t^{(2)}}\, dL_t^1 \right]=
\widehat \pi_1(y).
\end{align*}
An analogous relation holds for $\widehat p_2$ and $\widehat \pi_2$.
In other words, the boundary Laplace transforms $\widehat \pi_1$ and $\widehat \pi_2$ arise as the limits of $u\,\widehat p_1(u,y)$ and $u\,\widehat p_2(u,x)$ as $u\to 0$.
\end{rem}

This functional equation characterizing the invariant measure has been studied extensively. It was first derived in~\cite{dai_reflecting_2011}, has been used in \citeS{franceschi_kourkova_asymptotic_2017} to obtain asymptotic expansions, in \cite{franceschi_raschel_integral_2019} to derive explicit integral formulas, and in \cite{franceschi_bousquet_melou_price_hardouin_raschel_stationary_2023} to analyze differential properties.
Several variants of this equation have also been investigated. This includes the case of non-convex cones \citeS{franceschi_fayolle_raschel_fourier_transform_indagationes_2022,franceschi_fayolle_raschel_non_convex_mprf_2022}, as well as degenerate cases arising in interacting particle systems \citeS{franceschi_ichiba_karatzas_raschel_degenerate_2024,
franceschi_degenerate_asym_2025}.

\subsection{Transient case: Green's functions}

Recall that we denote by
$g(z_0,z)$ the Green's function, with $z=(z_1,z_2)$.
We introduce its Laplace transform in space by
\begin{equation}
\label{eq:Lap_green}
\widehat g (x,y)
= \mathbb{E}_{z_0} \!\left[ \int_0^\infty e^{x Z_t^{(1)}+y Z_t^{(2)} } \, dt \right]=
\int_{\mathbb R_+^2} e^{x z_1+yz_2} g(z_0,z)\,dz .
\end{equation}
For the sake of brevity, we omit the dependence on $z_0$, even though it is clearly present.
We also define the Laplace transforms of the boundary Green measures:
\begin{equation}
\label{eq:Lap_green1}
\widehat g_1(y)
=
\mathbb{E}_{z_0}\!\left[\int_0^\infty e^{y Z_t^{(2)}}\, dL_t^1\right]=
\int_{\mathbb R_+^2} e^{yz_2} g_1^{z_0}(z_0,z)\,dz ,
\end{equation}
and similarly for $\widehat g_2(x)$.
Note that the functions $\widehat g$, $\widehat g_1$, and $\widehat g_2$ coincide respectively with $\widehat p(0,x,y)$, $\widehat p_1(0,y)$, and $\widehat p_2(0,x)$.
To ensure convergence, we assume that $\Re x<0$ and $\Re y<0$. As before, we will see that these transforms admit analytic continuations to suitable domains of the complex plane.
The following proposition states the functional equation satisfied by these transforms.

\begin{prop}[Functional equation for the Green's functions] Assume that conditions \eqref{u11} are satisfied so that the process is transient.
For $x,y\in\mathbb{C}$ such that $\Re x<0$, $\Re y<0$, we have
\begin{equation}
\label{eq:Petit-24}
- K(x,y)\widehat g(x,y)
=
k_1(x,y)\widehat g_1(y)
+
k_2(x,y)\widehat g_2(x)
+
e^{\langle (x,y),z_0\rangle},
\end{equation}
where $K,k_1,k_2$ are defined in~\eqref{eq:def_parameters_Maxence}.
\end{prop}

\begin{proof}
This follows by applying the transient Basic Adjoint Relationship (BAR) (see Proposition~\ref{prop:bartransient}) to the test function $f(t,z)=e^{x z_1+y z_2}$, which satisfies the required assumptions for $z\in\mathbb{R}_+^2$, and $\Re x<0$, $\Re y<0$.
\end{proof}
\begin{rem}[From the functional equation for the transition density to the one for the Green's functions] Formally, the functional equation~\eqref{eq:Petit-24} is a particular case of the general equation~\eqref{eq:main_eq}, obtained by setting $u=0$.
\end{rem}
This functional equation characterizing the Green functions was derived in~\citeS{franceschi_green_2021} to obtain explicit integral formulas, and was further studied in~\cite{franceschi_kourkova_petit_asymptotics_2024} to derive asymptotic expansions, leading to the identification of the Martin boundary and the analysis of the associated harmonic functions.
Variants of this equation in related settings have also been investigated, notably in the case of a half-plane in~\citeS{franceschi_ernst_asymptotic_2021}, and for a killed space-time Brownian motion in~\citeS{franceschi_spacetime_2024}.

\subsection{Harmonic functions}

The PDEs satisfied by harmonic functions naturally lead, after Laplace transform, to functional equations which can be viewed as adjoint (or dual) to those obtained for the transition density, Green functions, and the invariant measure.

\paragraph{Escape along an axis}

Recall that the escape probability $h^\rightarrow(z)$ is harmonic for the reflected Brownian motion and satisfies~\eqref{eq:pdeharmneuman}. We introduce its Laplace transform in space:
\begin{equation}
\label{eq:Lap_h}
\widehat h^\rightarrow(x,y)
=
\int_{\mathbb R_+^2} e^{x z_1 + y z_2} h^\rightarrow(z)\,dz.
\end{equation}

We also define the boundary Laplace transforms:
\begin{equation}
\label{eq:Lap_h1}
\widehat h^\rightarrow_1(y)
=
\int_{\mathbb R_+} e^{y z_2} h^\rightarrow(0,z_2)\,dz_2,
\qquad
\widehat h^\rightarrow_2(x)
=
\int_{\mathbb R_+} e^{x z_1} h^\rightarrow(z_1,0)\,dz_1.
\end{equation}

Assuming $\Re x>0$ and $\Re y>0$ to ensure convergence, we obtain the following functional equation.

\begin{prop}[Functional equation for $h^\rightarrow$]
For $(x,y)\in\mathbb C^2$ such that $\Re x>0$ and $\Re y>0$, we have
\begin{equation}
\label{eq:FE_h_right}
K^*(x,y)\,\widehat h^\rightarrow(x,y)
=
k_1^*(x,y)\,\widehat h^\rightarrow_1(y)
+
k_2^*(x,y)\,\widehat h^\rightarrow_2(x)
+
c\,h^\rightarrow(0,0),
\end{equation}
where
\begin{equation}
\label{eq:def_parameters_adjoint}
    \left\{
    \begin{array}{rcl}
    K^*(x,y) & = & \frac{1}{2}(x^2+y^2+2\rho xy)-\mu_1x-\mu_2y,\\
    k_1^*(x,y) & = & \frac{1}{2} (x-r_1 y)+\rho y +\mu_1,\\
    k_2^*(x,y) & = & \frac{1}{2} (y-r_2 x)+\rho x +\mu_2,\\
    c & = & \frac{1}{2}(r_1+r_2)-\rho.
    \end{array}
    \right.
\end{equation}
\end{prop}

\begin{proof}
This follows by applying the adjoint generator $\mathcal{G}^*$ to the PDE~\eqref{eq:pdeharmneuman} satisfied by $h^\rightarrow$, and integrating against the exponential test function $e^{x z_1+y z_2}$. Integration by parts yields the boundary contributions, which produce the terms involving $\widehat h^\rightarrow_1$, $\widehat h^\rightarrow_2$, and the vertex term $h^\rightarrow(0,0)$.
\end{proof}
This functional equation was first derived and studied in~\cite{franceschi_fomichov_ivanovs_2022}.
The operator $K^*$ is the symbol of the adjoint generator, obtained from $K$ by reversing the drift. In that sense, this functional equation (and the one below) is adjoint to the functional equations obtained above.

\begin{rem}[A particular solution to the functional equation]
\label{rem:partsol}
Note that it is straightforward to verify that
\begin{equation}
\label{eq:fec}
K^*(x,y)\,\frac{1}{xy}
=
k_1^*(x,y)\,\frac{1}{y}
+
k_2^*(x,y)\,\frac{1}{x}
+
c.
\end{equation}
This is consistent with the fact that $h^\uparrow$ satisfies the same functional equation
\[
K^*(x,y)\,\widehat h^\uparrow(x,y)
=
k_1^*(x,y)\,\widehat h^\uparrow_1(y)
+
k_2^*(x,y)\,\widehat h^\uparrow_2(x)
+
c\,h^\uparrow(0,0),
\]
and that $h^\rightarrow + h^\uparrow = 1$. Summing the corresponding functional equations yields~\eqref{eq:fec} (we see the bivariate and univariate Laplace transforms of $1$).
Therefore, we have identified a particular solution to this functional equation. Subtracting $h^\uparrow(0,0)$ times~\eqref{eq:fec} from~\eqref{eq:FE_h_right} removes the constant term, thereby eliminating the inhomogeneity and simplifying the analysis. 
To characterize the solutions uniquely, one must further impose the boundary conditions, in particular the limiting values given in~\eqref{eq:limitvaluepra}.
\end{rem}

\paragraph{Escape to infinity and absorption at the vertex}

We now consider the escape probability to infinity $h^\infty(z)$, which is harmonic for the reflected Brownian motion killed at the origin. Its Laplace transform is defined by
\begin{equation}
\label{eq:Lap_h_inf}
\widehat h^\infty(x,y)
=
\int_{\mathbb R_+^2} e^{x z_1 + y z_2} h^\infty(z)\,dz,
\end{equation}
together with the boundary transforms
\begin{equation}
\label{eq:Lap_h_inf1}
\widehat h^\infty_1(y)
=
\int_{\mathbb R_+} e^{y z_2} h^\infty(0,z_2)\,dz_2,
\qquad
\widehat h^\infty_2(x)
=
\int_{\mathbb R_+} e^{x z_1} h^\infty(z_1,0)\,dz_1.
\end{equation}

Assuming $\Re x>0$ and $\Re y>0$, we obtain the following functional equation.

\begin{prop}[Functional equation for $h^\infty$]
For $(x,y)\in\mathbb C^2$ such that $\Re x>0$ and $\Re y>0$, we have
\begin{equation}
\label{eq:FE_h_inf}
K^*(x,y)\,\widehat h^\infty(x,y)
=
k_1^*(x,y)\,\widehat h^\infty_1(y)
+
k_2^*(x,y)\,\widehat h^\infty_2(x),
\end{equation}
where $K^*,k_1^*,k_2^*$ are defined in~\eqref{eq:def_parameters_adjoint}.
\end{prop}

\begin{proof}
This follows by applying the adjoint generator $\mathcal{G}^*$ to the PDE satisfied by $h^\infty$, together with the killing condition at the origin, and integrating against the exponential test function $e^{x z_1+y z_2}$. Integration by parts yields the boundary terms, while no contribution arises from the vertex due to the fact that $h^\infty(0)=0$.
\end{proof}
This functional equation was first derived in~\cite{franceschi_ernst_huang_escape_2021}, and was also studied in
\cite{franceschi_flin_sumofexponential_2024}.
Note that the functional equation satisfied by $h^0$ is
\[
K^*(x,y)\,\widehat h^0(x,y)
=
k_1^*(x,y)\,\widehat h^0_1(y)
+
k_2^*(x,y)\,\widehat h^0_2(x)
+
c.
\]
This is consistent with the fact that~\eqref{eq:fec} still holds and that $h^\infty + h^0 = 1$.

\section{Continuation on the Riemann surface and difference equation}
\label{sec:contRiem}

\subsection{The shifted kernel polynomial}
\label{sec:shiftk}

A recurring object is the \emph{shifted kernel} (a quadratic polynomial in two variables)
\[
K(x,y)+u=\tfrac12(x^2+y^2+2\rho xy)+\mu_1 x+\mu_2 y + u.
\]
This corresponds to the Laplace exponent of the driving Brownian motion with drift, shifted by the parameter $u$. The geometry of this kernel plays a central role in the analytical approach.

Solving the equation 
$$K(x,y)+u=0$$ 
for $x$ (resp.\ $y$) yields two algebraic branches, $X^\pm(y;u)$ (resp.\ $Y^\pm(x;u)$), which satisfy $$K(X^\pm(y;u),y)+u=0$$ (resp.\ $K(x,Y^\pm(x;u))+u=0$). The branch points of these branches are denoted by $y^\pm(u)$ (resp.\ $x^\pm(u)$).
We obtain the two branches
\[
X^\pm(y;u)=-(\rho y+\mu_1)\pm \sqrt{(\rho^2-1)y^2+2(\rho\mu_1-\mu_2)y+\mu_1^2 - 2u}.
\]
The branch points are the zeros of the discriminant, i.e.\ the polynomial under the square root, hence
\[
y^\pm(u)=\frac{-(\rho\mu_1-\mu_2)\pm\sqrt{\mu_1^2 -2\rho\mu_1\mu_2+\mu_2^2 + 2(\rho^2-1)u}}{\rho^2-1}.
\]
In the same way, we obtain
\[
Y^\pm(x;u)=-(\rho x+\mu_2)\pm \sqrt{(\rho^2-1)x^2 + 2(\rho\mu_2-\mu_1)x + \mu_2^2 - 2u},
\]
\[
x^\pm(u)=\frac{-(\rho\mu_2-\mu_1)\pm\sqrt{\mu_1^2 - 2\rho\mu_1\mu_2 + \mu_2^2 + 2(\rho^2-1)u}}{\rho^2-1}.
\]

\subsection{The kernel Riemann surface and its Galois automorphisms}

We now introduce the Riemann surface associated with the algebraic functions
$X(y;u)$ and $Y(x;u)$. Since $K(x,y)+u$ has degree two in each variable, the Riemann surface defined below 
has genus~$0$. Let us define
\[
\mathcal{S}_u=\{(x,y)\in\mathbb{C}^2:\;K(x,y)+u=0\}.
\]
The surface $\mathcal{S}_u$ is naturally endowed with two canonical involutive automorphisms $\eta$ and $\zeta$. 
These automorphisms are defined by exchanging the two roots of the kernel equation while keeping one coordinate fixed.
Using Vieta's formulas, for fixed $x$ the two roots $Y^\pm(x;u)$ satisfy
\[
Y^+(x;u)+Y^-(x;u)=-2(\rho x+\mu_2).
\]
Hence, if $(x,y)\in\mathcal{S}_u$, the second root $y'$ such that $K(x,y')+u=0$ is given by
\[
y'=-2(\rho x+\mu_2)-y.
\]
This defines the involution
\[
\zeta(x,y)
=\left(x,-2(\rho x+\mu_2)-y\right).
\]
Similarly, for fixed $y$ the two roots $X^\pm(y;u)$ satisfy
\[
X^+(y;u)+X^-(y;u)=-2(\rho y+\mu_1),
\]
which leads to the involution
\[
\eta(x,y)=\left(-2(\rho y+\mu_1)-x,\,y\right).
\]
By construction, if $(x,y)\in\mathcal{S}_u$, then
\[
K(\zeta(x,y))+u=K(\eta(x,y))+u=0,
\]
so that $\zeta$ and $\eta$ are automorphisms of $\mathcal{S}_u$.
Moreover, they are involutions:
\[
\zeta^2=\eta^2=\mathrm{Id}.
\]
The group generated by these two automorphisms,
\[
\langle \zeta,\eta \rangle,
\]
is thus a dihedral group, which plays a central role in the study of the model.
Remark that the involutions $\zeta$ and $\eta$ do not depend on the parameter $u$. 

\subsection{Uniformization}

The following uniformization holds:
\begin{equation}
\mathcal{S}_u=\{(x(s), y(s)): s \in \mathbb{C}^* \},
\label{eq:SC*}
\end{equation}
where, for $\beta$ defined in \eqref{eq:beta}, we have
$$
\left\{\begin{array}{l}
x(s)=\dfrac{x^{+}(u)+x^{-}(u)}{2}+\dfrac{x^{+}(u)-x^{-}(u)}{4}\left(s+\dfrac{1}{s}\right), \\[1.2ex]
y(s)=\dfrac{y^{+}(u)+y^{-}(u)}{2}+\dfrac{y^{+}(u)-y^{-}(u)}{4}\left(\dfrac{s}{e^{i \beta}}+\dfrac{e^{i \beta}}{s}\right).
\end{array}\right .
$$
(The above uniformization is established for $\mathcal{S}_0$ in~\cite{franceschi_kourkova_asymptotic_2017}; the proof is similar for $\mathcal{S}_u$.)

In particular, we see from~\eqref{eq:SC*} that $\mathcal{S}_u$ is conformally equivalent to $\mathbb{C}^*$. 
The parametrization $s \mapsto (x(s),y(s))$ provides a uniformization of $\mathcal S_u$, through which the automorphisms $\zeta$ and $\eta$ admit simple lifted representations.
More precisely, there exist transformations $\widehat\zeta$ and $\widehat\eta$ acting on $\mathbb{C}^*$ such that, for all $s \in \mathbb{C}^*$,
\[
\zeta\bigl(x(s),y(s)\bigr) = \bigl(x(\widehat\zeta(s)),y(\widehat\zeta(s))\bigr),
\qquad
\eta\bigl(x(s),y(s)\bigr) = \bigl(x(\widehat\eta(s)),y(\widehat\eta(s))\bigr),
\]
where
\[
\widehat\zeta(s)=\frac{1}{s},
\qquad
\widehat\eta(s)=\frac{e^{2i\beta}}{s}.
\]
Indeed, one checks that
\[
x\!\left(\tfrac{1}{s}\right)=x(s),
\qquad
y\!\left(\tfrac{1}{s}\right) = -2(\rho x(s)+\mu_2)-y(s),
\]
and that
\[
y\!\left(\tfrac{e^{2i\beta}}{s}\right)=y(s),
\qquad
x\!\left(\tfrac{e^{2i\beta}}{s}\right) = -2(\rho y(u)+\mu_1)-x(s).
\]
Thus, the Galois automorphisms of $\mathcal S_u$ correspond, via the uniformization, to the M\"obius transformations
$s \mapsto \frac{1}{s}$,
$s \mapsto \frac{e^{2i\beta}}{s}$,
acting on $\mathbb{C}^*$.

\begin{rem}[Riemann surface in the complex projective plane]
In fact, two points are missing in the definition of $\mathcal{S}_u$. To capture the entire Riemann surface, we should include points at infinity. This can be done by defining the surface in the complex projective plane as
\[
\mathcal{S}_u=\{K(x,y)+u=0\}\subset \mathbb{P}^2(\mathbb{C}),
\]
that is, the points $[x:y:z]\in \mathbb{P}^2(\mathbb{C})$ such that
\[
z^2\!\left(K\!\left(\tfrac{x}{z},\tfrac{y}{z}\right)+u\right)
=\tfrac12(x^2+y^2+2\rho xy)+\mu_1 xz+\mu_2 yz+u z^2=0.
\]
Setting $z=0$, and recalling that $\cos\beta=-\rho$, the kernel equation reduces to
\[
x^2-2\cos(\beta)\, xy+y^2=0,
\]
which is independent of $u$. Hence, the two points at infinity are
\[
P^\infty_\pm := [e^{\pm i\beta} : 1 : 0].
\]
Note that
\[
\lim_{s\to\infty} \frac{x(s)}{y(s)}=e^{i\beta},
\qquad
\lim_{s\to0} \frac{x(s)}{y(s)}=e^{-i\beta}.
\]
With this definition,
\[
\mathcal{S}_u
=\{[x(s): y(s):1]: s \in \mathbb{C}^*\}\cup\{P^\infty_+,P^\infty_-\}.
\]
Hence, $\mathcal{S}_u$ is conformally equivalent to the complex projective line
\[
\mathcal{S}_u \simeq \mathbb{P}^1 (\mathbb{C})=\mathbb{C}\cup\{\infty\}.
\]
It is also worth noting that the automorphisms swap the points at infinity:
$\eta(P^\infty_+)=P^\infty_-$ and $\zeta(P^\infty_+)=P^\infty_-$.
\end{rem}

\subsection{Universal covering and lifted automorphisms}

In view of the uniformization $\mathcal S \simeq \mathbb{C}^*$, it is natural to introduce the universal covering of $\mathcal S $.
More precisely, we define the map
\[
\begin{aligned}
\lambda:\widetilde{\mathcal S}\equiv\mathbb{C}&\longrightarrow \mathbb{C}^*\\
\omega&\longmapsto \lambda(\omega):=e^{i\omega}.
\end{aligned}
\]
This map is a $2\pi$-periodic, non-branching covering from $\mathbb{C}$ onto $\mathbb{C}^*$.
In particular, any horizontal segment of the form
\[
[a+ib,\,a+2\pi+ib],\qquad a,b\in\mathbb{R},
\]
is mapped onto a closed curve in $\mathbb{C}^*$ winding once around the origin.

Given $s\in\mathbb{C}^*$ (or more generally a subset $S\subset\mathbb{C}^*$), we denote by $\widetilde{s}$ (resp.\ $\widetilde{S}$) a chosen preimage under $\lambda$ in a prescribed vertical strip of width $2\pi$, typically $\{0\leq \Re\omega<2\pi\}$.
Every conformal automorphism $\chi$ of $\mathbb{C}^*$ admits a lifted automorphism
\[
\widetilde{\chi}=\lambda^{-1}\circ\chi\circ\lambda
\]
acting on $\mathbb{C}$.
Since $\lambda^{-1}$ is multivalued, this lift is uniquely determined once the image of a single point is fixed.
We now describe the lifts of the Galois automorphisms introduced above.
Recall that the lifted maps on $\mathbb{C}^*$ are
\[
\widehat\zeta(s)=\frac{1}{s},
\qquad
\widehat\eta(s)=\frac{e^{2i\beta}}{s}
\]
where the fixed points are given by 
$
s_1^-=e^{i\pi}$ and $s_2^-=e^{i(\pi+\beta)}$.
We define their lifts $\widetilde\zeta$ and $\widetilde\eta$ on $\mathbb{C}$ by requiring that
$\widetilde\zeta(\pi)=\pi$ and
$\widetilde\eta(\pi+\beta)=\pi+\beta$.
A direct computation then gives
\[
\widetilde\zeta(\omega)=-\omega+2\pi,
\qquad
\widetilde\eta(\omega)=-\omega+2(\pi+\beta).
\]
These transformations are central symmetries of respective centers $\widetilde{s}_1^-=\pi$ and $\widetilde{s}_2^-=\pi+\beta$.

As a consequence, their compositions are translations:
\[
\widetilde\eta\circ\widetilde\zeta(\omega)=\omega+2\beta,
\qquad
\widetilde\zeta\circ\widetilde\eta(\omega)=\omega-2\beta.
\]
Hence, the group generated by $\widetilde\zeta$ and $\widetilde\eta$ acts on $\mathbb{C}$ by affine transformations generated by reflections and translations, which provides a convenient framework for describing the analytic continuation on $\mathcal S$.


\subsection{Analytic continuation and difference equation}

We now describe how the functional equations allow us to extend analytically the unknown functions on the universal covering $\widetilde{\mathcal S}\simeq\mathbb{C}$.

\paragraph{Invariant measure}
Full details can be found in Appendix D of \cite{franceschi_bousquet_melou_price_hardouin_raschel_stationary_2023}.
We start from the functional equation
\[
-K(x,y)\widehat{\pi}(x,y)
=
k_1(x,y)\widehat{\pi}_1(y)
+
k_2(x,y)\widehat{\pi}_2(x).
\]
By construction, the functions $\widehat{\pi}_1$ and $\widehat{\pi}_2$ are initially defined in the half-planes
\[
\Delta_y=\{y\in\mathbb{C}:\ \Re y\leqslant 0\},
\qquad
\Delta_x=\{x\in\mathbb{C}:\ \Re x\leqslant0\},
\]
respectively. Using the uniformization $s=e^{i\omega}$, we lift these domains to the universal covering by defining
\[
\widetilde{\Delta}_y=\{\omega\in\mathbb{C}:\beta\leqslant \Re\omega<2\pi+\beta\text{ and }\ y(e^{i\omega})\in\Delta_y\},
\quad
\widetilde{\Delta}_x=\{\omega\in\mathbb{C}:0\leqslant \Re \omega<2\pi\text{ and }\ x(e^{i\omega})\in\Delta_x\}.
\]
On these domains, the compositions
\[
\widetilde{\pi}_1(\omega):=\widehat{\pi}_1\bigl(y(e^{i\omega})\bigr),
\qquad
\widetilde{\pi}_2(\omega):=\widehat{\pi}_2\bigl(x(e^{i\omega})\bigr)
\]
are well defined and analytic. 
Since $K\bigl(x(e^{i\omega}),y(e^{i\omega})\bigr)=0$, the functional equation then provides the following relation on $ \widetilde{\mathcal S}$, namely for $\omega\in\widetilde{\Delta}_x\cap\widetilde{\Delta}_y $, 
\begin{equation}
0
=
k_1(\omega)\,
\widetilde{\pi}_1(\omega)
+
k_2(\omega)\,
\widetilde{\pi}_2(\omega),
\label{fe:revetement}
\end{equation}
where we use the shorthand notations $\widetilde k_1(\omega)=k_1\bigl(x(e^{i\omega}),y(e^{i\omega})\bigr)$ and $\widetilde k_2(\omega)=k_2\bigl(x(e^{i\omega}),y(e^{i\omega})\bigr)$. This identity allows us to obtain an analytic continuation of $\widetilde{\pi}_1$ and $\widetilde{\pi}_2$ to the union
\[
\widetilde{\Delta}_x \cup \widetilde{\Delta}_y.
\]
Indeed, on the intersection of these domains, both $\widetilde{\pi}_1$ and $\widetilde{\pi}_2$ are well defined, and the functional equation provides a compatibility relation which allows analytic continuation across their common boundary.
We now use the action of the lifted automorphisms to extend further these functions.
We remark that, in the proper domains, invariance properties hold:
\begin{equation}
\widetilde{\pi}_1(\widetilde{\eta}(\omega))=\widetilde{\pi}_1(\omega)
\quad\text{and}\quad
\widetilde{\pi}_2(\widetilde{\zeta}(\omega))=\widetilde{\pi}_2(\omega)
.
\label{inv:relevement}
\end{equation}
We now derive directly on the universal covering $\widetilde{\mathcal S}\simeq\mathbb{C}$ the analytic continuation formula induced by the functional equation.
\begin{prop}[Meromorphic continuation and difference equation] The function $\widetilde\pi_1$ can be meromorphically continued to $\mathbb{C}$, through the difference equation
\begin{equation}
\widetilde\pi_1(\omega+2\beta)=\widetilde G(\omega) \widetilde\pi_1(\omega)
\label{eq:diffrel}
\end{equation}
where
\begin{equation}
\widetilde G(\omega):=\frac{\widetilde k_1(\omega) \widetilde k_2(\zeta (\omega))}{\widetilde k_2(\omega) \widetilde k_1(\zeta (\omega))} .
\label{tildeG}
\end{equation}
The functional equation \eqref{fe:revetement} and the invariance properties \eqref{inv:relevement} remain valid on the whole complex plane $\mathbb{C}$.
\end{prop}
\begin{proof}[Sketch of proof]
Recall that for all $\omega\in\widetilde{\Delta}_x \cup \widetilde{\Delta}_y$
\begin{equation}
\widetilde k_1(\omega)\,\widetilde{\pi}_1(\omega)
+
\widetilde k_2(\omega)\,\widetilde{\pi}_2(\omega)
=0.
\label{eq:FE_lifted}
\end{equation}
Furthermore, by~\eqref{inv:relevement}, in the proper domains, we have 
\begin{equation}
\widetilde k_1(\widetilde{\zeta}(\omega))
\,
\underbrace{\widetilde{\pi}_1(\widetilde{\zeta}(\omega))}_{\widetilde{\pi}_1(\widetilde{\eta}\circ\widetilde{\zeta}(\omega))}
+
\widetilde k_2(\widetilde{\zeta}(\omega))\,\underbrace{\widetilde{\pi}_2(\widetilde{\zeta}(\omega))}_{\widetilde{\pi}_2(\omega)}
=0.
\label{eq:FE_lifted_eta}
\end{equation}
Using \eqref{eq:FE_lifted} and \eqref{eq:FE_lifted_eta}, we eliminate $\widetilde{\pi}_2(\omega)$ and we get \eqref{eq:diffrel}
using
\[
\widetilde{\eta}\circ\widetilde{\zeta}(\omega)=\omega+2\beta.
\]
This relation provides a direct analytic continuation of $\widetilde{\pi}_2$ on the universal covering with the translation $\omega\mapsto\omega+2\beta$.
It just remains to check that
$$
\mathbb{C}=
\bigcup_{n\in\mathbb{Z}} \left(
\widetilde{\Delta}_x \cup \widetilde{\Delta}_y+2n\beta \right)
$$
which follows from the fact that $\widetilde{\Delta}_x \cup \widetilde{\Delta}_y$ contains the fundamental domain which is a strip of the form
\[
\{\omega\in\mathbb{C}:\ \alpha<\Re\omega<\alpha+2\beta\}.
\]
\end{proof}
A completely symmetric argument, yields an analogous relation for $\widetilde{\pi}_2$ which can be meromorphically continued to $\mathbb{C}$.

\paragraph{Green's function}
In the same way, we define
\[
\widetilde{g}_1(\omega):=\widehat{g}_1\bigl(y(e^{i\omega})\bigr),
\qquad
\widetilde{g}_2(\omega):=\widehat{g}_2\bigl(x(e^{i\omega})\bigr)
\]
and we have the functional equation and the invariance properties
\begin{equation}
0
=
k_1(\omega)\,
\widetilde{g}_1(\omega)
+
k_2(\omega)\,
\widetilde{g}_2(\omega)
+e^{x(e^{i\omega})a+y(e^{i\omega})b},
\label{feg:revetement}
\end{equation}
\begin{equation}
\widetilde{g}_1(\widetilde{\eta}(\omega))=\widetilde{g}_1(\omega)
\quad\text{and}\quad
\widetilde{g}_2(\widetilde{\zeta}(\omega))=\widetilde{g}_2(\omega)
.
\label{invg:relevement}
\end{equation}
In the same way, we can then prove the following statement.
\begin{prop}[Meromorphic continuation and difference equation for the Green's function]
The function $\widetilde g_1$ can be meromorphically continued to $\mathbb{C}$ through the difference equation
\begin{equation}
\widetilde g_1(\omega+2\beta)
=
\widetilde G(\omega)\,\widetilde g_1(\omega)
+
\widetilde H(\omega),
\label{eq:diffrel_green}
\end{equation}
where
\[
\widetilde G(\omega)
:=
\frac{k_1(\omega)\,k_2(\widetilde\zeta(\omega))}
{k_2(\omega)\,k_1(\widetilde\zeta(\omega))}
\]
and
\begin{equation}
\widetilde H(\omega)
:=
\frac{k_2(\widetilde\zeta(\omega))}{k_1(\widetilde\zeta(\omega))
} \left(
\frac{
\,e^{x(e^{i\omega})a+y(e^{i\omega})b}}{k_2(\omega)}
-
\frac{\,e^{x(e^{i\widetilde\zeta(\omega)})a+
y(e^{i\widetilde\zeta(\omega)})b}
}{
\,k_2(\widetilde\zeta(\omega))} \right) .
\label{tildeH}
\end{equation}
The functional equation \eqref{feg:revetement} and the invariance properties \eqref{invg:relevement} remain valid on the whole complex plane $\mathbb{C}$.
\end{prop}

\paragraph{Harmonic functions}
The harmonic functions $\widehat h^\infty_1(y)$ and $\widehat h^\rightarrow_1(y)$ can also be lifted to the universal covering and satisfy similar difference equations, replacing $K$ by $K^*$ and the $k_i$ by $k_i^*$. We obtain an equation of the form
$$
\widetilde h^\infty_1(\omega+2\beta)=\widetilde G^*(\omega) \widetilde h^\infty_1(\omega).
$$
Note that, due to Remark~\ref{rem:partsol}, considering $\widehat h^\rightarrow_1(y)-c\,\widehat h^\rightarrow_1(0)$ yields a similar equation without inhomogeneous term.

\paragraph{Transition density}

Similarly, it can be shown that the Laplace transform of the transition density satisfies a non-homogeneous difference equation similar to that of Green's functions. It suffices to replace the kernel $K(x,y)$ by the shifted kernel $K(x,y)+u$. This equation has not yet been studied in the literature.

\subsection{Degenerate covariance}


When $\det\Sigma=0$, i.e. $\rho=\pm 1$ (degenerate case), the situation simplifies. The kernel $K$ becomes \emph{parabolic}, and the associated curve degenerates. 
The geometry becomes simpler: there is no need to pass to a universal covering since the Riemann surface is conformally equivalent to $\mathbb{C}$ (rather than $\mathbb{C}^*$) and is then simply connected. The analytic continuation can be performed directly on the curve itself.
More precisely, the Riemann surface is uniformized by $\mathcal{S}=\{(x(s), y(s)): s \in \mathbb{C} \}$ where
$$
\left\{\begin{array}{l}
x(s)=2 s\left(s+\mu_2\right), \\[1.2ex]
y(s)=2 s\left(s-\mu_1\right).
\end{array}\right .
$$
The automorphisms $\widehat \zeta$ and $\widehat \eta$ of the surface $\mathcal{S}$ (through the uniformization) are defined by
$$
\widehat\zeta s:=-s-\mu_2, \quad\quad \widehat\eta s:=-s+\mu_1 \text { and then } \widehat\eta\zeta s = s +\mu_1+\mu_2.
$$
We now state the difference equation obtained in \cite{franceschi_degenerate_asym_2025}. Its iteration provides a meromorphic continuation of $\widehat \pi_1 (y(s))$ to all of $\mathbb{C}$.
\begin{prop}[Meromorphic continuation and difference equation]
The function $\widehat \pi_1$ admits a meromorphic continuation to the whole complex plane $\mathbb{C}$ and satisfies the difference equation
\[
\widehat \pi_1(y(s+(\mu_1+\mu_2)))=\widehat G(s)\,\widehat \pi_1(y(s)),
\]
where
\[
\widehat G(s):=\frac{k_1(s)\,k_2(\zeta s)}{k_2(s)\,k_1(\zeta s)}
=\frac{(s-s_1)(s+s_2+\mu_2)}{(s-s_2)(s+s_1+\mu_2)}.
\]
Moreover, the functional equation and the invariance properties by $\widehat\eta$ and $\widehat\zeta$ extend to the whole complex plane $\mathbb{C}$.
\end{prop}


%

\section{Integral representations via Carleman boundary value problems}
\label{sec:intrep}

\subsection{From the functional equation to a Carleman BVP}

To establish the Carleman boundary value problem, we need to introduce the curve
\begin{equation}
\label{eq:curve_definition1}
     \R =\{y\in\mathbb C: K(x,y)=0 \text{ and } x\in(-\infty,x^-)\}=Y^\pm ((-\infty,x^-))
\end{equation}
where $x^-=x^-(0)$ is defined in Section~\ref{sec:shiftk}. 
The curve $\R$ is a branch of a hyperbola symmetric with respect to the real axis.
We denote by $\G$ the open domain in $\mathbb{C}$ bounded by $\R$ and containing $0$, see Figure~\ref{BVPtheta1}.

\begin{figure}[hbtp]
\centering
\includegraphics[trim = 0cm 0.3cm 0cm 0.3cm, scale=0.75]{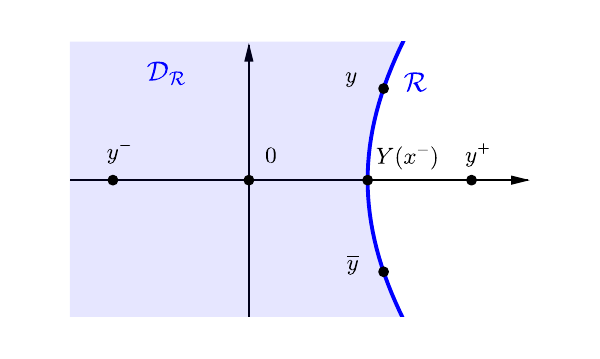}
\caption{
The curve $\R$ in~\eqref{eq:curve_definition1} is symmetric with respect to the horizontal axis, and $\G$ is the blue domain}
\label{BVPtheta1}
\end{figure}

We now state the boundary value problem that characterizes the stationary distribution of reflected Brownian motion in a quadrant. This result is established in \citeS{franceschi_raschel_integral_2019}.
Define, for $y\in \R$,
\begin{equation}
\label{intro:eq:G}
     G(y)=\frac{k_1}{k_2}(X^-(y),y)\frac{k_2}{k_1}(X^-(y),\overline{y}).
\end{equation}     
One can observe that the lift of $G$ to the universal covering of the Riemann surface is $\widetilde G$, defined in \eqref{tildeG}. 
\begin{lemm}[Carleman Boundary Value Problem for the invariant measure]
\label{prop:BVP_general0}
The Laplace transform $\widehat \pi_1$ satisfies the following boundary value problem:
\begin{enumerate}[label={\rm (\roman{*})},ref={\rm (\roman{*})}]
     \item\label{item:1_BVP0_general} $\widehat \pi_1$ is meromorphic on $\G$ with at most a simple pole $p$ and is bounded at infinity;   
     \item\label{item:2_BVP0_general} $\widehat \pi_1$ is continuous on $\overline{\G}\setminus \{p\}$ and
\begin{equation}
\label{eq:boundary_condition_general1}
     \widehat \pi_1(\overline{y})=G(y)\widehat \pi_1({y}), \qquad \forall y\in \R.
\end{equation}
\end{enumerate}
\end{lemm} 
\begin{rem}[BVP on the universal covering]
The proof is thus similar to that of the difference equation \eqref{eq:diffrel}. The difference equation can be viewed as the lift of this boundary value problem to the universal covering of the Riemann surface. It is also possible to consider a fundamental domain and to formulate the boundary value problem directly on the universal covering, as Malyshev does in the discrete setting in his seminal article \cite{malysev_analytic_1972}. We do not adopt this approach here, as it was not the one followed in \citeS{franceschi_raschel_integral_2019}, although it would have been equally relevant and elegant.
\end{rem}

The boundary value problem satisfied by the Green function is established in \citeS{franceschi_green_2021} in the case where the process is transient. 
Recall that the functional equation in the transient case
differs from the recurrent case by an additional term $e^{(x,y) \cdot z_0}$, where $z_0\in\mathbb{R}_+^2$ is the starting point of the process.
This equation leads to a non-homogeneous boundary problem, which is more complex to solve than in the recurrent case. 
Recall that $G$ is defined in~\eqref{intro:eq:G}
and define $H$ such that
$$
H(y)=
\frac{k_1}{k_2} {(X^-(y),\overline{y})}
\left(
\frac{e^{(X^-(y),y) \cdot z_0}}{k_2(X^-(y),y)}
-
\frac{e^{(X^-(y),\overline{y})\cdot z_0}}{k_2(X^-(y),\overline{y})}
\right) .
$$
One can observe that the lift of $H$ to the universal covering of the Riemann surface is $\widetilde H$, defined in \eqref{tildeH}.
\begin{theorem}[Non-homogeneous Carleman boundary problem for the Green function]
The Laplace transform $\widehat g_1^{z_0}$ of the Green function satisfies the following boundary problem:
\begin{enumerate}[label={\rm (\roman{*})},ref={\rm (\roman{*})}]
\item $\widehat g_1^{z_0}$ is meromorphic on $\G$ and continuous on $\overline{\G}$;
     \item $\widehat g_1^{z_0}$ satisfies the boundary condition
\begin{equation*}
     \widehat g_1^{z_0}(\overline{y})=G(y)\widehat g_1^{z_0}({y}) + g(y), \qquad \forall y\in \mathcal{R}.
\end{equation*}
\end{enumerate}
\end{theorem}

\subsection{Conformal gluing functions}

To construct a particular solution to the boundary value problems, we look for a conformal function $w$ that glues together the upper and lower parts of the hyperbola $\R$.
We define $w$ in terms of the generalized Chebyshev polynomials
\begin{equation*}
T_a(x)  =\cos (a\arccos (x))=\frac{1}{2} \Big\{\big(x+\sqrt{x^2-1}\big)^a+\big(x-\sqrt{x^2-1}\big)^a\Big\}
\end{equation*} as follows:
\begin{equation}
\label{eq:expression_CGF_BM1}
     {w} (y)=T_{\frac{\pi}{\beta}}\bigg(-\frac{2y-(y^++y^-)}{y^+-y^-}\bigg),
\end{equation}
where we recall that
$\beta$
is the cone angle defined in \eqref{eq:beta}.

\begin{lemm}[Conformal gluing function]
The function $w$ satisfies the following properties.
\begin{enumerate}[label={\rm (\roman{*})},ref={\rm (\roman{*})}]
     \item\label{item:conformal_10} $w$ is analytic on $\G$, continuous on $\overline{\G}$, and unbounded at infinity,
     \item\label{item:conformal_20} $w$ is injective on $\G$, 
     \item\label{item:conformal_30} $w(y)=w(\overline{y})$ for all $y\in\mathcal{R}$.
\end{enumerate}
\end{lemm} 
The conformality property follows from~\ref{item:conformal_10} and~\ref{item:conformal_20}, and the gluing property follows from~\ref{item:conformal_30}. The proof can be found in  \citeS{franceschi_raschel_tutte_2017}.

\subsection{Integral formulas}


To solve these Carleman BVP, a standard technique is to map the contour $\mathcal{R}$ to a segment via a \emph{conformal gluing function} $w$ (CGF), i.e.\ a conformal map such that $w(y)=w(\overline{y})$ for $y\in\mathcal{R}$.
This reduces \eqref{eq:boundary_condition_general1} to a classical Riemann(-Hilbert) BVP on an interval, solvable by Sokhotski–Plemelj formulas. 
Solving this problem provides an explicit integral expression for $\widehat{\pi}_1$ and $\widehat g_1^{z_0}$, see \citeS{franceschi_raschel_integral_2019} and \citeS{franceschi_green_2021} for more details. In particular, all the various constants are determined there.

\begin{theorem}[Explicit formula for the Laplace transform] 
\label{thm:mainexplicit}
For all $y\in \G$, the Laplace transform $\widehat \pi_1$ is given by
\begin{equation}
\label{eq:main_formula_with_constants1}
\widehat \pi_1(y)
=C  \left( \frac{w(0)-w(p)}{w(y)-w(p)} \right)^{-\chi}
 \exp\left(\frac{1}{2i\pi} \int_{\R^-} \log G(t) \left[ \frac{w'(t)}{w(t)-w(y)}
\right]
\mathrm{d}t\right),
\end{equation}
where $C>0$ is an explicit constant,
$p$ is the possible pole of $\widehat \pi_1$,
$\chi$ equals $0$ or $-1$ depending on the parameters, $\R^-$ is the half with negative imaginary part of the hyperbola $\R$ defined in \eqref{eq:curve_definition1}, and $w$ is the conformal gluing function defined in~\eqref{eq:expression_CGF_BM1}.
\end{theorem}



The formula for the Green's function is more complicated because it involves solving a non-homogeneous boundary value problem.
\begin{theorem}[Integral formula for Green functions]
The Laplace transform $\widehat g_1^{z_0}$ can be written as follows:
$$
\widehat g_1^{z_0} (y)=
\frac{-F (w(y))}{2i\pi} \int_{\R^-} \frac{g(t)}{Y^+ (w(t))}
\left(
 \frac{w'(t) }{w(t)-w(y)}
 + \chi
  \frac{ w'(t) }{w(t)} 
  \right)
  \mathrm{d}t
$$
where $\chi=0$ or $1$, and
$$
F (w(y))=w(y)^{\chi} 
\exp  \left(
\frac{1}{2i\pi} \int_{\R^-} \log(G(t))\left(\frac{ w'(t)}{w(t)-w(y)}- \frac{ w'(t)}{w(t)} \right) \mathrm{d}t
\right),
$$
where $Y^+$ is the boundary value on $\R$ of the function $X$ defined on $\G$.
\end{theorem}

\paragraph{Escape and absorption probabilities}
Similar integral formulas for $\widehat{h}_1^\infty$ and $\widehat{h}_1^\uparrow$ can be found in \cite{franceschi_ernst_huang_escape_2021} and \cite{franceschi_fomichov_ivanovs_2022}.

\paragraph{Degenerate cases: hyperbola $\to$ parabola}
In degenerate models (rank-one covariance), the same philosophy applies.
The kernel curve is parabolic, the analytic continuation and the uniformization is simpler in some respects. The hyperbola $\R$ becomes a parabola and similar integral formulas exist. That said, in this degenerate case, it is possible to solve the general case wihtout Carlemann BVP and directly with Tutte's invariant method or with the compensation approach using special functions of theta-type, see the next sections. One can refer to \cite{franceschi_degenerate_asym_2025} for a detailed treatment in the context of gap processes of a degenerate particle system.

\section{Explicit (integral-free) formulas via Tutte invariants}
\label{sec:explifor}

\subsection{Decoupling and invariants}
\label{sec:foncdecinv}

From the 1970s to the 1990s, Tutte developed an algebraic approach based on certain invariants to solve a functional equation arising in the study of the enumeration of well-colored triangulations \cite{Tutte-95}. 
This approach has recently been extended to other contexts, proving particularly fruitful. Tutte invariants were first used to determine the algebraic nature of solutions to functional equations. An analytic version later appeared, refining and deepening the analytic method of \citet{fayolle_random_2017}. Recent applications include planar maps \cite{albenque_menard_schaeffer,BeBM-11}, walks confined to the quarter plane \cite{bernardi_counting_2015}, and, for the first time in the continuous framework, reflected Brownian motion \citeS{franceschi_raschel_tutte_2017,franceschi_bousquet_melou_price_hardouin_raschel_stationary_2023}, as well as interacting particle systems \citeS{franceschi_ichiba_karatzas_raschel_degenerate_2024,franceschi_degenerate_asym_2025}.

Tutte's invariants method is based on the following steps.
We need to look for \emph{invariants}, that is, (meromorphic) functions $I$ such that $$I(Y^+)=I(Y^-).$$ 
On the Riemann surface, the invariants are the functions $I$ that are invariant under $\eta$ and $\zeta$.
Invariants can be divided into two main categories:
\begin{itemize}
\item an \emph{explicit canonical invariant} (e.g., the conformal gluing function),
\item an \emph{unknown invariant} depending on generating functions or Laplace transforms involved in the problem.
\end{itemize} 
To find unknown invariants it is possible to look for \emph{decoupling functions}. A decoupling function $D$ is simply a particular solution to the boundary condition or the difference equation that is rational (or sometimes analytic). These functions are additive or multiplicative:
\begin{itemize}
\item if $F$, the function we are looking for, satisfies $F(Y^+)=F(Y^-)+H(x,Y^+,Y^-),$ an {additive decoupling function} is of the form 
     $$H(x,Y^+,Y^-)=D(Y^+)-D(Y^-)$$ so that the function $F-D$ is an invariant;
\item if $F$, the function we are looking for, satisfies $F(Y^+)=G(x,Y^+,Y^-)F(Y^-),$ a {multiplicative decoupling function} is of the form 
     $$G(x,Y^+,Y^-)=D(Y^-)/D(Y^+)$$ so that the function $D\cdot F$ is an invariant.
\end{itemize}
Then, we use \emph{invariant lemmas}, typically stating that there are few invariants and that, if two are found, one is a rational function of the other (this rational function is determined by the study of poles and zeros of the invariants).

To be applicable, this method thus requires the existence of two types of functions, invariants and decoupling functions. 
There are cases where these functions do not exist and the method is inapplicable.
When it works, this method makes it possible to obtain explicit expressions (without integrals) for the Laplace transforms (which also allows to obtain information on the differential and algebraic nature of these functions).

\subsection{Integral-free formulas}

The article \citeS{franceschi_bousquet_melou_price_hardouin_raschel_stationary_2023} identifies the parameters for which the formula \eqref{eq:main_formula_with_constants1} simplifies dramatically. 
We have seen that the difficulty in solving the boundary problem \eqref{eq:boundary_condition_general1} lies in the fact that $G\neq1$. Using decoupling functions, the goal is therefore to reformulate the boundary problem by reducing to the case ``$G=1$''.
For a rational invariant $D$, the function $G$ defined in \eqref{intro:eq:G} can be written as the ratio
\begin{equation}
\label{eq:G->F}
     G(\thetadeux) = \frac{D(\thetadeux)}{D(\overline{\thetadeux})}.
\end{equation}
This would allow us to rewrite the boundary condition \eqref{eq:boundary_condition_general1} as
\begin{equation*}
     (D\cdot\widehat \pi_1)(\overline{\thetadeux})=(D\cdot\widehat \pi_1)({\thetadeux}),
     \quad 
     \forall y\in\R
\end{equation*}
where $D\cdot\widehat\pi_1$ is meromorphic and is an invariant. This reduces the problem to a boundary value problem solvable by a simple invariant method. In this case, $D\cdot\widehat\pi_1$ can be expressed as a rational fraction in the conformal gluing function $w$, which is the canonical invariant. This is the Tutte invariant method . Sometimes one only finds functions $D$ such that
\begin{equation}
\label{eq:doubledecoupl}
     G^2(\thetadeux) = \frac{D(\thetadeux)}{D(\overline{\thetadeux})}
\end{equation}
(we say that $D$ is a decoupling function of degree two) and then it is $D\cdot\widehat\pi_1^2$ which is an invariant.

This analysis leads to the following definition and lemma \citeS{franceschi_bousquet_melou_price_hardouin_raschel_stationary_2023}. Refer to Figure~\ref{fig:linear_transformation1} for a reminder of the angles $\beta,\delta,\epsilon$, and $\vartheta$.
\begin{lemm}[Decoupling functions]
\label{def:decoupling_function}
We say that $D$ is a decoupling function of order 1 if and only if $D$ is rational and satisfies \eqref{eq:G->F}. Such a function exists if and only if
\begin{equation}
\epsilon+\delta\in \beta \mathbb{Z}+ \pi\mathbb{Z}.
\label{cnsdecou}
\end{equation}
We say that $D$ is a decoupling function of order 2 if $D$ is rational and satisfies \eqref{eq:doubledecoupl}. Such a function exists if and only if
\begin{equation}
2\epsilon+\vartheta-\beta \text{ and } 2\delta-\vartheta \in \beta \mathbb{Z}+ \pi\mathbb{Z}.
\label{cns2decou}
\end{equation}
\end{lemm}

In comparison with the discrete case \cite{bernardi_counting_2015}, which is solved through case-by-case analysis, it is quite remarkable to obtain a unified necessary and sufficient condition in the continuous setting despite the many parameters (drift, covariance, reflection vectors). 
Note also that the decoupling functions are rational and may have very high degree, whereas in the case of random walks they are typically of degree $2$ or~$3$. 
\begin{theorem}[Integral-free formula for the invariant measure]
If \eqref{cnsdecou} (resp. \eqref{cns2decou}) is satisfied, then there exists an explicit rational function $A\in\mathbb C(X,Y)$ such that the Laplace transform $\widehat \pi_1$ is
\begin{equation}
\label{eq:main_result_sketched}
    \widehat \pi_1 (\thetadeux)=A(\thetadeux,w(\thetadeux))
    \quad (\text{resp.} \quad
     \widehat  \pi_1 (\thetadeux)=\sqrt{A}(\thetadeux,w(\thetadeux)))
.
\end{equation} 
\end{theorem}
The previous proposition can be found in~\citeS{franceschi_bousquet_melou_price_hardouin_raschel_stationary_2023} with more detail. In particular, the rational function $A$ is made explicit there.

\begin{remarque}[\textit{Skew} symmetry, sum of exponentials, and orthogonal reflections]
Some special cases of this theorem allow us to reprove classical results. For instance, when $\epsilon+\delta=\pi$, we recover the case of \textit{skew} symmetry, which characterizes the situations where the invariant measure has a product form and is exponential, see~\citet{harrison_multidimensional_1987}.
This also allows us to recover the result of \citet{dieker_reflected_2009}, which characterizes, via the condition $\alpha\in-\mathbb{N}$, the cases where the stationary density is a sum of exponentials. Finally, we also recover the case of orthogonal reflections \citeS{franceschi_raschel_tutte_2017}, where $\epsilon=\delta=\beta$.

\begin{figure}[hbtp]
\centering
\includegraphics[scale=1.2]{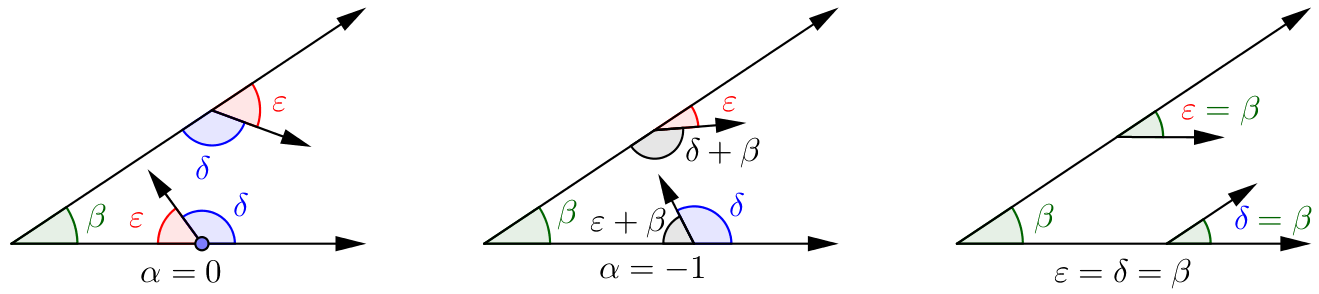}
\caption{Reflection angles in three well-known cases. From left to right: \textit{skew} symmetry, Dieker and Moriarty condition, and orthogonal reflections}
\label{fig:three_examples}
\end{figure}
\end{remarque}

\paragraph{Degenerate case}
Papers \cite{franceschi_ichiba_karatzas_raschel_degenerate_2024} and \cite{franceschi_degenerate_asym_2025} apply this method in the degenerate case.
These papers use a \emph{non rational decoupling} function to solve in all cases the difference equation.
To state the results we need to define
$$
s_1:=\frac{r_1 \mu_1-\mu_2}{1+r_1} \quad \text { and } \quad s_2:=\frac{\mu_1-r_2 \mu_2}{1+r_2}
$$
and the (meromorphic) decoupling function
$$
D(s):=\frac{\Gamma\left(s-s_1\right) \Gamma\left(s+s_2+\mu_2\right)}{\Gamma\left(s-s_2\right) \Gamma\left(s+s_1+\mu_2\right)} 
$$
where $\Gamma$ is the gamma function.
\begin{thm}[Explicit expression of $\widehat\pi_1$: degenerate case]
There exists a rational function $R \in \mathbb{C}(X)$ such that the Laplace transform $\phi_1$ satisfies
$$
\widehat\pi_1(y)=D\left(\frac{1}{2}\left(\mu_1+\sqrt{2 y+\mu_1^2}\right)\right) R\left(\tan \left(\frac{\pi}{2} \sqrt{2 y+\mu_1^2}\right)\right) .
$$
An explicit expression depending for $R$ is given in \cite{franceschi_degenerate_asym_2025}.
\end{thm}
These papers derive explicit boundary invariant measures in terms of theta-like functions acted on by polynomial differential operators.

\paragraph{Green's functions: towards a double decoupling}
For Green's functions in the transient regime, the presence of the inhomogeneous term (the exponential term) in the functional equation suggests combining additive and multiplicative decouplings (a ``double decoupling'') to obtain explicit representations and differential classifications.
This is an ongoing work.

\paragraph{Escape and absorption probabilities}
The Tutte's invariant method was applied in the paper \cite{franceschi_flin_sumofexponential_2024} to study the absorption probability at the apex of the cone, $h^0$, and to determine the cases in which it takes the form of a sum of exponentials.

\section{Differential and algebraic classification via difference Galois theory}
\label{sec:diffgalois}

\paragraph{Historical background}

In the 19th century, Liouville, Picard and Vessiot initiated a Galois theory for differential equations, later formalized by Kolchin. In parallel, the Galois theory of difference equations studies functional equations algebraically and connects relations between solutions and coefficients; see \cite{vdps}. Typical examples include finite difference, $q$-difference, and Mahler equations.

After early developments in the 20th century and a revival in the 1990s due to links with several areas (combinatorics, special functions, etc.), this theory has seen significant generalizations; see \cite{dreyfus_hdr}. It now enables, in particular, the study of differential algebraic properties of solutions and applications to random walks in the quarter plane \cite{DreyfHardtderiv,deryfus_hardouin_roques_singer,DHRS0}. To our knowledge, \citeS{franceschi_bousquet_melou_price_hardouin_raschel_stationary_2023} is the first to apply these ideas to a continuous process such as Brownian motion.


\subsection{An algebraic-differential hierarchy}
\label{sec:algdiffhier}

One of our objectives is to study the Laplace transform of the stationary distribution of reflected Brownian motion in a cone. Using the analytic method, we obtained an explicit (complicated) expression involving integrals and various trigonometric and algebraic functions, see Theorem~\ref{thm:mainexplicit}. 

When the parameter $\alpha$ is a non-positive integer (under a non-degeneracy condition), Dieker and Moriarty~\cite{dieker_reflected_2009} showed that the stationary density is a finite sum of exponentials of the form
$\sum_{i} c_i e^{-a_i x- b_i y}$.
This implies that the Laplace transform is a rational function. This significant simplification raises the following natural question:

\smallskip
{\itshape For which parameter values does the Laplace transform have a simple expression?}

\smallskip
\noindent Before answering this question, we must clarify what we mean by \emph{simple}. We can classify the complexity of a function using the following natural hierarchy:
\beq\label{hierarchy}
\text{rational} \subset \text{algebraic} \subset \text{D-finite} \subset \text{D-algebraic}.
\eeq
A function is said to be:
\begin{itemize}
\item \emph{rational}, if it is a ratio of polynomials.
\item \emph{algebraic}, if it satisfies a polynomial equation with coefficients in the field of rational functions over $\RR$.
\item \emph{differentially finite} (D-finite or \emph{holonomic}), if it satisfies a linear differential equation with rational function coefficients over $\RR$.
\item \emph{differentially algebraic} (D-algebraic), if it satisfies a polynomial differential equation with rational function coefficients over $\RR$.
\item \emph{differentially transcendental} (or D-transcendental), if it is not D-algebraic.
\end{itemize}

\begin{exe}[Nature of the conformal gluing function $w$]
\label{exemple:naturew}
If \( \frac{\pi}{\beta} \in \mathbb{Z} \), then the function \( w \) introduced in~\eqref{eq:expression_CGF_BM1} is a polynomial.  
If \( \frac{\pi}{\beta} \in \mathbb{Q} \backslash \mathbb{Z} \), then the function \( w \) is algebraic but not a rational function.  
If \( \frac{\pi}{\beta} \notin \mathbb{Q} \), then the function \( w \) is D-finite but not algebraic. If \( \frac{\pi}{\beta} \notin \mathbb{Q} \), then the function \( 1/w \) is differentially algebraic but not D-finite.
\end{exe}

A major advance of \cite{franceschi_bousquet_melou_price_hardouin_raschel_stationary_2023} is a sharp classification of the Laplace transform of the stationary distribution (non-degenerate wedge case) into this hierarchy
with explicit necessary and sufficient conditions given by linear relations between the \emph{angles} of the model (wedge opening, reflection angles, drift direction).

The algebraic and differential properties of the Laplace transform affect the stationary distribution in various ways. Let us illustrate this with two examples.

\begin{itemize}  
\item[]{Implications for \em{moments}:} 
if the Laplace transform is:
\begin{itemize}
\item {D-algebraic}, then the differential equation it satisfies translates into a recurrence relation for the moments $M_n$ of the stationary distribution. In general, this relation is of infinite order with polynomial coefficients in \( n \);
\item {D-finite}, the recurrence becomes linear and of finite order.
\end{itemize}
\item[]{Implications for the \em{density}:}
if the Laplace transform is:
\begin{itemize}
\item {D-finite}, then the density is also D-finite;
\item {Rational}, the density can be written as a linear combination of terms of the form \( x^k e^{-ax} \), with \( k \in \mathbb{N}_0 \).  
\end{itemize} 
\end{itemize}

This type of study, presented here through the example of a continuous stationary distribution, was until now mostly limited to discrete combinatorial problems. These results linking the nature of generating functions to properties of their coefficients are well known in analytic combinatorics, see the reference book by Flajolet and Sedgewick \cite{Flajolet_Sedgewick_2009} or the book by Mishna \cite{Mishna2020}.


\subsection{Galoisian criteria for differential transcendence}
\label{subsec:difftranscriteria}

The Galois theory of difference equations studies algebraic relations between the solutions $g_0,\ldots,g_n$ of linear difference equations of the form
\beq\label{eq:shift}
\sigma(g_i)=g_i +b_i,
\eeq
for $0\le i \le n$, where the coefficients $b_i$ lie in a field $K$ equipped with an automorphism $\sigma$. For example, we can take $K=\mathbb{C}(z)$ and $\sigma(g)(z)=g(z+h)$ or $g(qz)$ or $g(z^p)$ depending on whether we deal with a finite difference equation, a $q$-difference equation, or a Mahler equation.

In particular, a theorem by Ostrowski (in the context of differential rather than difference equations~\cite{Ostrowski}) gives necessary and sufficient conditions for the algebraic independence of $g_0,\ldots,g_n$ over $K$ in terms of algebraic relations satisfied by the coefficients $b_i$.

This framework also allows studying the differential algebraicity of a function $g$ satisfying
$$\sigma(g)=g+b,$$ provided the derivation $\partial$ commutes with $\sigma$. Indeed, the functions $g_i=\partial ^i g$ then satisfy a system of the form~\eqref{eq:shift}, with $b_i=\partial ^i b$, and are algebraically dependent if and only if $g$ satisfies a differential equation of order at most $n$. We then say that $g$ is $\partial$-algebraic.

Using the theorem proven by Dreyfus and Hardouin in~\cite{DreyfHardtderiv} (Thm.~C.8, case $\Delta=0$), we then obtain the following corollary stated in \citeS{franceschi_bousquet_melou_price_hardouin_raschel_stationary_2023} (Theorem 9.2).
\begin{corfr}[Application]
  \label{cor:abstractdiffalgcriterianonhomog}
 Let $K=\mathbb{C}(z)$ with the usual derivation and the field automorphism~$\sigma$ defined by $\sigma (f) (z)=f(z+1)$.
Then the fixed field of $\sigma$ is $K^\sigma=\mathbb{C}$ and one can take $L$ to be the set of meromorphic functions on $\mathbb{C}$.
 Let $f\in L$ be a non-zero meromorphic function satisfying 
 $$f(z+1)= a(z)f(z) $$ 
 for some $a \in \mathbb{C}(z)$. 
 If $f$ is differentially algebraic over $\mathbb{C}(z)$, then there exist $N \in \N_0$,
 constants $c_0,\dots,c_N \in \mathbb{C}$, not all zero, and $h \in \mathbb{C}(z)$ such that
 \begin{equation}\label{eq:equationDA}
   c_0 \frac{\partial a }{a}(z) +c_1 \partial \left(\frac{\partial a }{a}\right) (z) +\dots +c_N \partial^N\left( \frac{\partial a}{a}\right)(z)= h(z+1)-h(z).
 \end{equation} 
\end{corfr}

We now introduce the notion of divisor.
\begin{deffr}[Divisors]
A divisor on $\mathbf{P}^1(\mathbb{C}) = \mathbb{C} \cup \{\infty\}$ is a finite formal expression of the form $\sum n_a[a]$ where each $n_a \in \mathbb{Z}$ and $a$ ranges over elements of $\mathbf{P}^1(\mathbb{C})$. The support of a divisor is the finite set of all $a$ such that $n_a \neq 0$. The divisor $\operatorname{div}(f)$ of a rational function $f \in \mathbb{C}(z)^*$ is defined as
\[
\operatorname{div}(f) := \sum_a \operatorname{ord}_a(f)[a]
\]
where $\operatorname{ord}_a(f)$ denotes the order of $f$ at the point $a$, and the sum is taken over all $a \in \mathbf{P}^1(\mathbb{C})$ and is in fact finite.
\end{deffr}

Consider the automorphism $$\sigma (f) (z)=f(z+1)$$ and define its action on divisors by
\[
\sigma\left(\sum_a n_a\left[a\right]\right) := \sum_a n_a\left[a -1\right].
\]
It is then clear that 
\begin{equation}
\label{eq:sdiv}
\sigma(\operatorname{div}(f)) = \operatorname{div}(\sigma(f)).
\end{equation}
An analogous notion of elliptic divisor also exists relative to $\sigma_q f(z) = f(qz)$. The following two lemmas are useful for the study of the decoupling functions introduced in~\ref{sec:foncdecinv}.

\begin{lem}[Lemma 2.1 \cite{vdps}]
Let $g \in \mathbb{C}(z)^*$, then $g$ can be written as 
$$g(z)=\sigma(f)(z) f^{-1} (z)=\frac{f(z+1)}{f(z)}$$ 
for some $f \in \mathbb{C}(z)^*$ if and only if the following three conditions hold:
\begin{enumerate}
    \item $\infty$ is not in the support of $\operatorname{div}(g)$;
    \item $g(\infty) = 1$;
    \item For every $\mathbb{Z}$-invariant orbit $E$ in $\mathbb{C}$, that is, a subset of $\mathbb{C}$ of the form $e + \mathbb{Z}$, we have
    \[
    \sum_{a \in E} \operatorname{ord}_a(g) = 0.
    \]
\end{enumerate}
\end{lem}
The converse is also true and its proof is constructive, yielding an explicit $f$ given $g$ (see Section 2.1 of \cite{vdps}). This construction often inspires how decoupling functions are built.

Here is another lemma that may prove useful when applying this method (and which, to the best of our knowledge, does not appear in this form in the existing literature).
\begin{lemm}
For $g\in\mathbb{C}(z)^*$ and $h\in\mathbb{C}(z)$, the equation~\eqref{eq:equationDA} of Corollary~\ref{cor:abstractdiffalgcriterianonhomog} 
\[
c_0 \frac{\partial g }{g}(z) +c_1 \partial \left(\frac{\partial g}{g}\right) (z) +\dots +c_N \partial^N\left( \frac{\partial g}{g}\right)(z)= h(z+1)-h(z)
\]
with at least one $c_k \neq 0$, implies that for any $\mathbb{Z}$-invariant orbit $E$ in $\mathbb{C}$ we have
\[
\sum_{a \in E} \operatorname{ord}_a(g) = 0.
\]
\end{lemm}
\begin{proof}
This follows from a partial fraction decomposition. Let $\operatorname{div}(g)=\sum m_a [a]$, i.e., $m_a=\operatorname{ord}_a(g)$. The left-hand side of~\eqref{eq:equationDA} becomes
\[
\sum_{k=0}^N c_k \partial^k \left( \frac{\partial g}{g} \right)(z)=\sum_{k=0}^N c_k (-1)^k k!\sum_a \frac{m_a}{(x-a)^k} .
\]
Decompose the right-hand side of~\eqref{eq:equationDA} into partial fractions. Let $d_k^a$ be the coefficients (possibly zero), then
\[
h(z+1)-h(z)=\sum_a \sum_k \frac{d_k^a}{(z+1-a)^k}  - \sum_a \sum_k \frac{d_k^a}{(z-a)^k}=\sum_a\sum_k \frac{d_k^{a-1}-d_k^a}{(z-a)^k}.
\]
Choose a non-zero $c_k$, and by uniqueness of the partial fraction decomposition, for every $a$ we have
\[
c_k (-1)^k k! m_a = d_k^{a-1}-d_k^a,
\]
thus
\[
c_k (-1)^k k! \sum_{a\in E} m_a 
=\sum_{a\in E} (d_k^{a-1}-d_k^a)=\sum_{a\in E-1} d_k^a -\sum_{a\in E} d_k^a =0
\]
since $E-1=E$.
\end{proof}
Lemma 3.8 of \cite{Hardouin2008} provides an analogous result by replacing $\sigma$ with $\sigma_q$ defined by $\sigma_q f(z)=f(qz)$ and replacing the notion of divisor with that of elliptic divisor.

\subsection{Differential and algebraic classification}

Tutte’s invariants method, presented in the previous section, provides sufficient conditions for a functional equation's solution to be differentially algebraic.
To find the necessary conditions, we use Galois theory, which provides the criteria for differential transcendence presented above. 
Here is a result obtained in \citeS{franceschi_bousquet_melou_price_hardouin_raschel_stationary_2023}.
\begin{theorem}[Algebraic nature of the Laplace transform]
\label{thm:main_diffalg}
If $\frac{\beta}{\pi}\notin \mathbb Q$, then $\widehat\pi_1$ is differentially algebraic over $\mathbb C(\thetadeux)$ if and only if:
\begin{equation}
\label{eq:CNS}
   \epsilon+\delta\in \beta \mathbb{Z}+ \pi\mathbb{Z}
   \quad \text{or}
   \quad
   (2\epsilon+\vartheta-\beta
\text{ and }   
   2\delta-\vartheta \in \beta \mathbb{Z}+ \pi\mathbb{Z}).
\end{equation}
If $\frac{\beta}{\pi}\in \mathbb Q$, then $\phi_1$ is algebraic over $\C(\thetadeux)$ if and only if condition \eqref{eq:CNS} is satisfied.
\end{theorem}
The article \citeS{franceschi_bousquet_melou_price_hardouin_raschel_stationary_2023} also established necessary and sufficient conditions for $\widehat\pi_1$ to belong to each category in the function hierarchy introduced earlier, summarized in Table~\ref{table:characterisation}. 
Let
\[
\alpha = \frac{\varepsilon+\delta-\pi}{\beta}, 
\qquad 
\alpha_1 = \frac{2\varepsilon+\theta-\beta-\pi}{\beta}, 
\qquad 
\alpha_2 = \frac{2\delta-\theta-\pi}{\beta}.
\]
Note how remarkably compact and geometric they are.
\renewcommand{\arraystretch}{1.5}
\begin{table}[h!]
\begin{center}
\begin{tabular}{|c|c|c|c|c|}
\hline 
  & D-algebraic & D-finite & Algebraic & Rational  \\ 
\hline 
  $\beta/\pi \notin \Q$ &
 $\alpha \in \Z +\frac\pi \beta\Z$, or 
 & $\alpha \in -\N +\frac\pi \beta\Z$, or
 &  $\alpha \in -\N$, or  & $\alpha \in -\N$ 
  \\
  &{ $\{\alpha_1, \alpha_2 \} \subset \Z +\frac\pi \beta\Z$ }
 &{ $\{\alpha_1, \alpha_2 \} \subset   \Z \cup \left(-\N+ \frac \pi\beta \Z\right)$}
  &{ $\{\alpha_1, \alpha_2 \} \subset \Z $ } &\\
  \hline
  $\beta/\pi \in \Q$  & 
 always & $\alpha \in \Z +\frac\pi \beta\Z$, or 
     & $\alpha \in \Z +\frac\pi \beta\Z$, or  
    & $\alpha  \in  -\N$
  \\
  &&{ $\{\alpha_1, \alpha_2 \} \subset \Z +\frac\pi \beta\Z$ }
   &{ $\{\alpha_1, \alpha_2\}\subset \Z +\frac\pi \beta\Z$ }&\\
 \hline      
\end{tabular}
\end{center}
\caption{Nature of the Laplace transform $\widehat \pi_1$: non degenerate case}
\label{table:characterisation}
\end{table}
In a sense, this article
completes the search for “simple” cases by identifying and listing them all, and by providing unified and simple explicit expressions for the Laplace transform in these cases.

\paragraph{Degenerate case}

In the degenerate case, the classification is obtained in \cite{franceschi_degenerate_asym_2025} and expressed in terms of the geometric parameters of the model. Let
$$\gamma=\frac{r_1r_2-1}{(1+r_1)(1+r_2)},
\quad
\gamma_1 = \mu_1 - \frac{2}{1+r_1},
\quad
\text{and}
\quad
\gamma_2 = \mu_2 - \frac{2}{1+r_2}.
$$
These are parameters similar to $\alpha$, $\alpha_1$ and $\alpha_2$.
The Table~\ref{table:characterisation2} shows the classification of the Laplace transform.
\renewcommand{\arraystretch}{1.5}
\begin{table}[h!]
\begin{center}
\begin{tabular}{|c|c|c|c|c|}
\hline 
 & D-algebraic & D-finite &  Algebraic & Rational  \\ 
\hline 
  $r_1$ and   $r_2 \neq 1$ &
 $\gamma\in\Z$, or $\{\gamma_1 ,\gamma_2 \}\subset \Z$
 & \multicolumn{3}{c|}{$\gamma \in -\N$}
  \\
  \hline
    $r_1$ or $r_2=1$  & 
 $\gamma_2$ or $\gamma_1\in \mathbb{N}$ & \multicolumn{3}{c|}{never }
 \\
 \hline      
\end{tabular}
\end{center}
\caption{Nature of the Laplace transform $\widehat \pi_1$: degenerate case}
\label{table:characterisation2}
\end{table}

\section{Asymptotics via saddle point method and singularity analysis}
\label{sec:asymptmartin}

\subsection{Tauberian-type transfer Lemmas}

Asymptotic studies are most often based on analytic continuations of generating functions or Laplace transforms and on the study of their singularities on the Riemann surface. This makes it possible to apply Tauberian-type transfer lemmas that deduce the asymptotic behavior of a function based on the type of singularity (pole or branch point).

The following propositions give an overview of the many possible versions of such lemmas. See, for example, the book by \citet[Theorems 35.1 and 37.1]{doetsch_introduction_1974} or the article by \citet[Lemmas C.1 and C.2]{dai_reflecting_2011} for more precise statements and proofs of the two lemmas below.

Let $f$ be a positive, continuous, and integrable function on $[0, \infty)$, and define its Laplace transform $g$ as follows:
$$
g(z)=\int_0^{\infty} e^{z x} f(x) d x, \quad \Re z<\alpha_0
$$
where $\alpha_0=c_p(g)(\equiv \sup \{\theta \geq 0 : g(\theta)<\infty\})$. The point $\alpha_0$ is the smallest singularity of $g(z)$, and $g(z)$ is analytic for $\Re z<\alpha_0$. We still denote by $g$ its analytic continuation when it exists.

\begin{lemme}[Transfer lemma for a pole of order $k$] 
\label{pr}
Let $g$ be an analytic continuation of the Laplace transform of $f$ such that
$$g(z) - \frac{c_0}{(\alpha_0 -z)^k}$$
is analytic for $\Re z < \alpha_1$ with $\alpha_1 > \alpha_0$. Suppose furthermore that for some constants $a, b, \delta>0$,
we have 
$$\quad|g(z)|<\frac{a}{|z|^{1+\delta}}, \quad \Re z \in\left[0, \alpha_1\right],|\Im z|>b$$
(this condition can be weakened, see \cite[Appendix C]{dai_reflecting_2011}).
Then, denoting $\Gamma$ as the gamma function,
$$ f(x)  \underset{x \to \infty}{\sim} \frac{c_0}{\Gamma (k)} x^{k-1} e^{- \alpha_0 x}.  $$
\end{lemme}
Define the set
$$\mathcal{G}_\delta\left(\alpha_0\right) =\left\{z \in \mathbb{C}: z \neq \alpha_0, \mid \arg (z-\left.\alpha_0\right) \mid>\delta\right\},$$ 
where $\arg z \in(-\pi, \pi)$ is the principal value of the argument of the complex number $z$.
\begin{lemme}[Transfer lemma for a branch point of order $\lambda$] 
\label{prbr}
Let $g$ be an analytic continuation of the Laplace transform of $f$ such that there exists a point $\alpha_0$ and an angle $\delta \in\left[0, \frac{1}{2} \pi\right)$ such that $g(z)$ is analytic on the set $\mathcal{G}_\delta$.
Suppose that $g(z) \rightarrow 0$ as $|z| \rightarrow \infty$ for $z \in \mathcal{G}_\delta(\alpha)$,
and that there exist $a \in \mathbb{R}, \lambda \in \mathbb{R}$ and $c_0 \in \mathbb{R}$ such that
$$
\lim _{\substack{z \rightarrow a \\ z \in \mathcal{G}_\delta(\alpha)}}(\alpha_0-z)^\lambda(g(z)-a)=c_0
$$
Then:
$$ f(x)  \underset{x \to \infty}{\sim} \frac{c_0}{\Gamma (\lambda)} x^{\lambda-1} e^{- \alpha_0 x}.  $$
\end{lemme}

\subsection{Uniform method of steepest descent}

The method of steepest descent is a procedure for determining the asymptotic behavior of integrals of the form
$$
I(r)=\int_C g(z) e^{r f(z)} d z
$$
where $f(z)$ and $g(z)$ are analytic functions, $r$ is a positive parameter tending to infinity, and $C$ is a contour in the complex plane. It was introduced by Debye in the early 20th century in an article on Bessel functions of large order. See the classic book by \citet{fedoryuk_asymptotic_1989}. The general idea is to deform the contour $C$ into a new integration path $C^{\prime}$ such that $C^{\prime}$ passes through one (or more) saddle point(s), i.e., points $z$ where $f^{\prime}(z)=0$, and follows the path of steepest descent, i.e., a curve where the imaginary part of $f(z)$ is constant.

To understand this, set $z=x+i y$ with $x$ and $y\in\mathbb{R}$ and
$$
f(z)=u(x, y)+i v(x, y)
$$
with $u$ and $v\in\mathbb{R}$,
and suppose that $z_0=x_0+i y_0$ is a simple zero of $f^{\prime}(z)$, which implies that $f^{\prime \prime}\left(z_0\right) \neq 0$. One can then classically show that $(x_0, y_0)$ is a saddle point of the function $u(x, y)$. The level curve given by $v(x, y)=v\left(x_0, y_0\right)$ thus determines the path of steepest descent of the surface defined by $z=u(x, y)$ for $(x, y, z)\in\mathbb{R}^3$. 

By the complex Morse lemma, there exists a bijective and holomorphic function on a neighborhood of $0$ such that $\omega\mapsto z(\omega)\in\mathbb{C}$, with $z(0)=z_0$ and
$$
f(z(\omega))=f\left(z_0\right)- \omega^2.
$$
It then follows easily that
$$
z'(0)=\sqrt{\frac{2}{-f^{\prime \prime}\left(z_0\right)}}.
$$
Thus, on the path of steepest descent, for $t\in\mathbb{R}$ we have
$f(z(\omega))=f\left(z_0\right)-t^2.$ 
Suppose that this formula is valid from $-\infty$ to $+\infty$ (otherwise the integral should be split around the saddle point and the remainder shown to be negligible). Making a change of variables from $z$ to $t$, we then obtain as $r$ tends to infinity,
\begin{align*}
I(r)&=e^{r f\left(z_0\right)} \int_{-\infty}^{\infty} g(z) \frac{d z}{d t} e^{-r t^2} d t
\\ &\sim
 g\left(z_0\right) e^{r f\left(z_0\right)} \sqrt{\frac{2}{-f^{\prime \prime}\left(z_0\right)}} 
\underbrace{\int_{-\infty}^{\infty} e^{-r t^2} d t}_{\sqrt{\pi/r}}.
\end{align*}

\paragraph{Saddle point depending on a parameter}
In the work we are interested in on bivariate asymptotics, the function $f$ and thus the saddle point depend on a parameter. We then use a uniform method of steepest descent with respect to this parameter. This is possible thanks to an extension of the Morse lemma, which we could not find in the classical literature and which is proved in Appendix A of the article \citeS{franceschi_kourkova_petit_asymptotics_2024}.

When there are several different saddle points, for example $z_+(\alpha)$ and $z_-(\alpha)$, which coincide at a particular value $\alpha_0$, there exist specific methods to determine the uniform asymptotic expansion near $\alpha_0$, see Wong’s chapter in the book \cite{Wong2001}.

\subsection{Stationary distribution asymptotics (non-degenerate)}

This section presents the asymptotic results obtained in the article \citeS{franceschi_kourkova_asymptotic_2017}.
There, we solve a difficult problem raised in the article by \citet[\S 8]{dai_reflecting_2011}, namely the computation of the asymptotics of $\mathbb{P}[Z_\infty \in  x c+B]$ as $x\to \infty$, 
where $ c \in {\mathbb{R}}_+^2$
is a directional vector
and $B \subset {\mathbb{R}}_+^2$ is an arbitrary compact set. More precisely, we determine the full asymptotic expansion of the density
$\pi(x_1,x_2)$ of $\Pi$  as $x_1,x_2 \to \infty$  and $x_2/x_1 \to  {\rm \tan } (\alpha)$, for any given angle $\alpha \in (0,\pi/2)$.
To obtain this finer result, we analytically continue the Laplace transforms on the kernel Riemann surface and then apply transfer theorems and the saddle point method with parameter. This method was developed in the discrete case by Malyshev~\cite{malyshev_asymptotic_1973}.

To state this result, we introduce some notation via Figure \ref{intro:fig:ellipse}.
Recall that the automorphisms $\zeta$ and $ \eta$ leave one of the two coordinates invariant.
The points $\eta \theta^*$ and $\zeta \theta^{**}$ are poles of the Laplace transforms.
For a given angle $\alpha \in [0,\pi/2]$, define the point $\theta(\alpha)$ on the ellipse ${\mathscr E}:=\{ \theta\in\mathbb{R}^2:\gamma(\theta)=0\}$, see Figure~\ref{intro:fig:ellipse}, by
\begin{equation}
\label{thetaalpha11}
\theta(\alpha)={\rm argmax}_{\theta \in {\mathcal E}}\langle \theta \vert e_\alpha  \rangle, \qquad \text{where }e_\alpha=(\cos \alpha, \sin \alpha).
\end{equation} 
This point is the saddle point used in the proof of the theorem below.

\begin{figure}[hbtp]
\centering 
\includegraphics[scale=0.6]{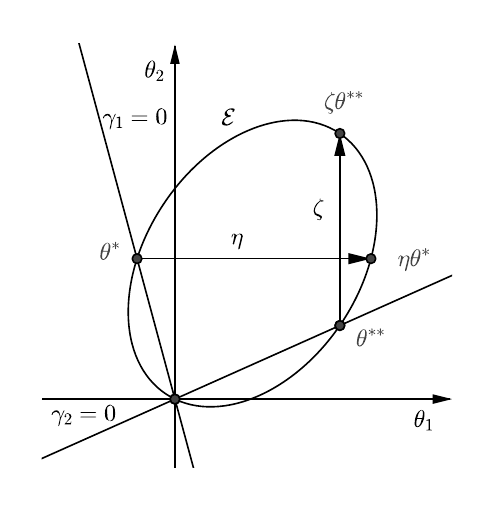} 
\includegraphics[scale=0.6]{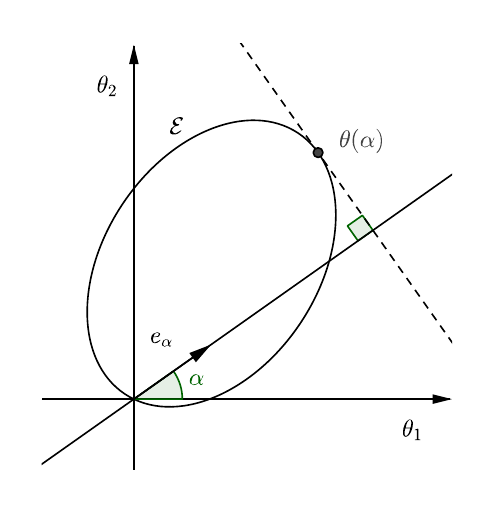} 
\caption{Left: depiction of the ellipse $\mathscr{E}:=\{ \theta\in\mathbb{R}^2:\gamma(\theta)=0\}$, the lines $\{k_1(\theta)=0\}$, $\{k_2(\theta)= 0\}$, and the points 
$\theta^{*}$, $\theta^{**}$, $\eta \theta^*$ and $\zeta \theta ^{**}$. Right: geometric interpretation of the saddle point $\theta(\alpha)$ on $\mathscr{E}$}
 \label{intro:fig:ellipse}
\end{figure}

The following result provides the leading term in the asymptotic expansion of $\pi(r\cos \alpha, r \sin \alpha)$ as $r\to\infty$ and $\alpha\to\alpha_0$, see Figure \ref{intro:fig:asympt}.
\begin{figure}[hbtp]
\centering
 \includegraphics[scale=0.4]{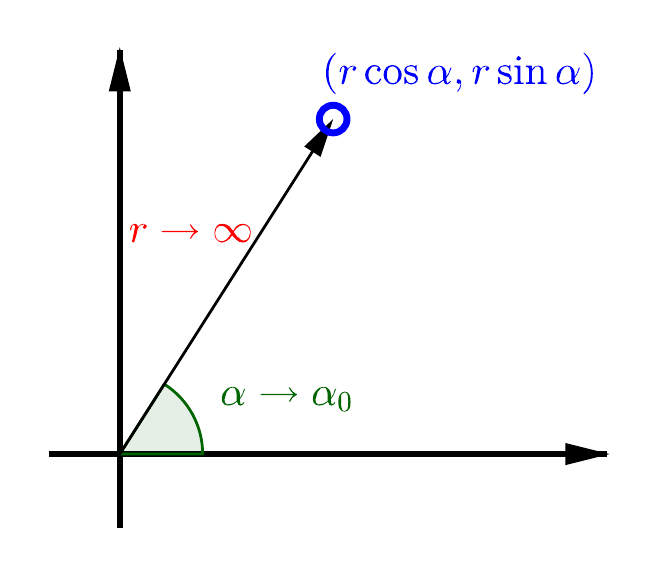}
 \vspace{-0.5cm}
\caption{Illustration of the asymptotic direction in Theorem \ref{intro:thm:asympt}}
\label{intro:fig:asympt}
\end{figure}

\begin{theorem}[Asymptotics of the invariant measure]
\label{intro:thm:asympt}
   Let $e_\alpha=( \cos \alpha, \sin \alpha)$ with $\alpha_0 \in (0, \pi/2)$. 
Then as $r \to \infty$ and $\alpha\to\alpha_0$ we have
\begin{equation}
\pi (r e_\alpha) = 
  (1+o(1))\cdot  \left\{
\begin{array}{ll}
\frac{C_0}{\sqrt{r}}e^{-r \langle e_\alpha \vert\theta(\alpha) \rangle}  & \text{ in } \mathscr{Q}_{--},\\
C_1e^{-r \langle e_\alpha \vert\eta\theta^{*}\rangle} & \text{ in } \mathscr{Q}_{+-},\\
C_2 e^{-r \langle e_\alpha \vert\zeta\theta^{**}\rangle}& \text{ in } \mathscr{Q}_{-+},\\
C_1e^{-r \langle e_\alpha \vert\eta\theta^{*}\rangle}  +
 C_2e^{-r \langle e_\alpha \vert\zeta\theta^{**}\rangle} & \text{ in } \mathscr{Q}_{++},
\end{array}
\right.
\end{equation}
  where $C_0$, $C_1$ and $C_2$ are constants and the sets $\mathscr{Q}_{\pm\pm}$ are explicit parameter-dependent regions.
\end{theorem}

The exponents in the integrals stem either from the poles $\eta\theta^*$ and $\zeta\theta^{**}$ or from the saddle point $\theta (\alpha)$. The proof consists of inverting the two-variable Laplace transform and uses standard techniques from complex analysis such as the residue theorem to reduce to the case of a simple integral. This leads to typical integrals for applying the saddle point method on the Riemann surface. There is thus a competition between the poles and the saddle point.

\begin{proof}[Key steps of the proof]
The first step is to meromorphically continue the Laplace transforms.
Using the functional equation and the inversion formula for Laplace transforms,
the density $\pi(x_1,x_2)$ can then be expressed as a double integral.
We then reduce to a single integral using the residue theorem.
\begin{align*}
\pi(z_1,z_2)
&=\frac{-1}{(2\pi i)^2} \int_{-i\infty}^{i\infty} \int_{-i\infty}^{i\infty}
e^{-z_1 x - z_2 y}
\frac{k_1(x,y)\widehat \pi_1(y)+k_2(x,y)\widehat \pi_2(x)}{\gamma(x,y)}
\,\mathrm{d} x \,\mathrm{d} y 
\\
&= \frac{1}{2\pi i} \int_{-i\infty}^{i\infty}
\widehat \pi_2(x)\,k_2\!\bigl(x, Y^+(x)\bigr)\, e^{-z_1 x - z_2 Y^+(x)}
\frac{\mathrm{d} x}{\sqrt{ d(x) }} 
\\
&\quad+\frac{1}{2\pi i} \int_{-i\infty}^{i\infty}
\widehat \pi_1(y)\,k_1\!\bigl(X^+(y), y\bigr)\, e^{-z_1 X^+(y) - z_2 y}
\frac{\mathrm{d} y}{\sqrt{ \widetilde d (y) }}.
\end{align*}
These integrals are typical for applying the saddle point method. 
The coordinates of the saddle point are the critical points of the functions
\begin{equation*}
\cos(\alpha)x+\sin(\alpha)Y^+(x)\quad \text{ and }\quad \cos(\alpha)X^+(y)+\sin(\alpha)y.
\end{equation*}
The saddle point is then the point $\theta(\alpha)$ defined in equation 
\eqref{thetaalpha11}. 
When deforming the integration contour, the poles of the above integrands, which turn out to be the points $\eta\theta^*$ and $\zeta\theta^{**}$, and their residues must be taken into account. 
The leading term in the asymptotics is then determined by the pole if the contour crosses a pole when deformed, and by the saddle point otherwise.
\end{proof}



\subsection{Green's functions asymptotics (non-degenerate)}

The article \citeS{franceschi_kourkova_petit_asymptotics_2024} determines the asymptotic behaviour of the Green’s functions of the reflected Brownian motion in a cone and derives Martin’s boundary from this.

\begin{figure}[hbtp]
\centering
\includegraphics[scale=0.7]{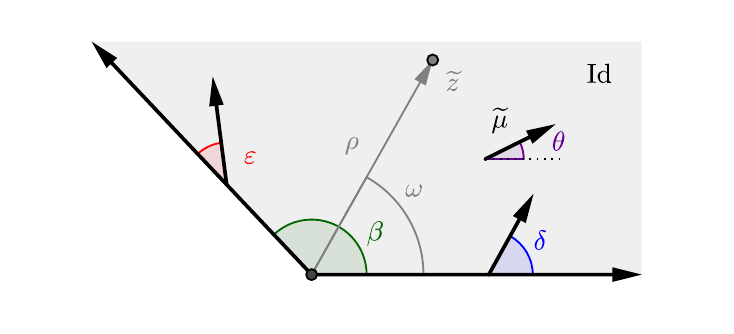}
\vspace{-0.5cm}
\caption{The cone with angle $\beta$, reflection angles $\delta$ and $\varepsilon$, and drift $\widetilde{\mu}$ with direction $\theta$. In gray, the point $\widetilde z$ with polar coordinates $\rho$ and $\omega$.}
\label{fig:cone}
\end{figure}


We use the symbol~$\sim$ to denote an asymptotic expansion of a function. If for some functions $f$ and $g_k$ we write $f(x)\sim \sum_{k=1}^n g_k(x)$ as $x\to x_0$, this means $g_k(x)=o(g_{k-1}(x))$ and $f(x)- \sum_{k=1}^n g_k(x)=o(g_n(x))$ as $x\to x_0$.

We now define the angles
\begin{equation*}
\omega^*:=\theta-2\delta\quad\text{and}\quad 
\omega^{**}:=\theta+2\epsilon .
\label{eq:omega***}
\end{equation*}
We can observe that $$\omega^{*}<\theta<\omega^{**}.$$
The following results are obtained using the analytic approach and the saddle-point method. The points $\omega^*$ and $\omega^{**}$ correspond, in a certain sense, to poles of the Laplace transforms of the Green functions, and the direction $\omega$ corresponds to the saddle point when we invert the Laplace transform.

\begin{theorem}[Asymptotics in the general case]
The Green function ${g}(\rho \cos \omega,\rho \sin \omega)$ has the following asymptotics as $\omega\to\omega_0\in(0,\beta)$ and $\rho\to\infty$, for any $n\in\mathbb{N}$:
\begin{itemize}
\item If $\omega^{*}<\omega_0<\omega^{**}$ then
\begin{equation}
{g}
(\rho \cos \omega,\rho \sin \omega)
\underset{\rho\to\infty \atop\omega\to\omega_0}{\sim}
e^{-2\rho|\widetilde{\mu}| \sin^2 \left( \frac{\omega-\theta}{2} \right)} \frac{1}{\sqrt{\rho}}
  \sum_{k=0}^n \frac{\widetilde{c_k}(\omega)}{ \rho^{k}}
  \label{eq:asymptsaddlepoint}
  \end{equation}
\item If $\omega_0<\omega^*$ then
\begin{equation}
{g}
(\rho \cos \omega,\rho \sin \omega)
\underset{\rho\to\infty \atop\omega\to\omega_0}{\sim}
c^{*} e^{-2\rho|\widetilde{\mu}| \sin^2 \left( {\omega+\delta-\theta} \right)}
+
e^{-2\rho|\widetilde{\mu}| \sin^2 \left( \frac{\omega-\theta}{2} \right)} \frac{1}{\sqrt{\rho}}
  \sum_{k=0}^n \frac{\widetilde{c_k}(\omega)}{ \rho^{k}}
  \label{eq:asymptpole1}
  \end{equation}
\item If $\omega^{**}<\omega_0$ then
\begin{equation}
{g}
(\rho \cos \omega,\rho \sin \omega)
\underset{\rho\to\infty \atop\omega\to\omega_0}{\sim}
c^{**}e^{-2\rho|\widetilde{\mu}| \sin^2 \left( {\omega-\epsilon-\theta} \right)}
+
e^{-2\rho|\widetilde{\mu}| \sin^2 \left( \frac{\omega-\theta}{2} \right)} \frac{1}{\sqrt{\rho}}
  \sum_{k=0}^n \frac{\widetilde{c_k}(\omega)}{ \rho^{k}}
  \label{eq:asymptpole2}
  \end{equation}
\end{itemize}
where $c^*$ and $c^{**}$ are {positive} constants and $c_k(\omega)$ are $\omega$-dependent coefficients such that $\widetilde{c_k}(\omega)\underset{\omega\to\omega_0}{\longrightarrow} \widetilde{c_k}(\omega_0)$.
\label{thm1}
\end{theorem}
There are four possible cases, illustrated in Figure~\ref{fig:coneasympt}.
\begin{figure}[hbtp]
\centering
\includegraphics[scale=0.85]{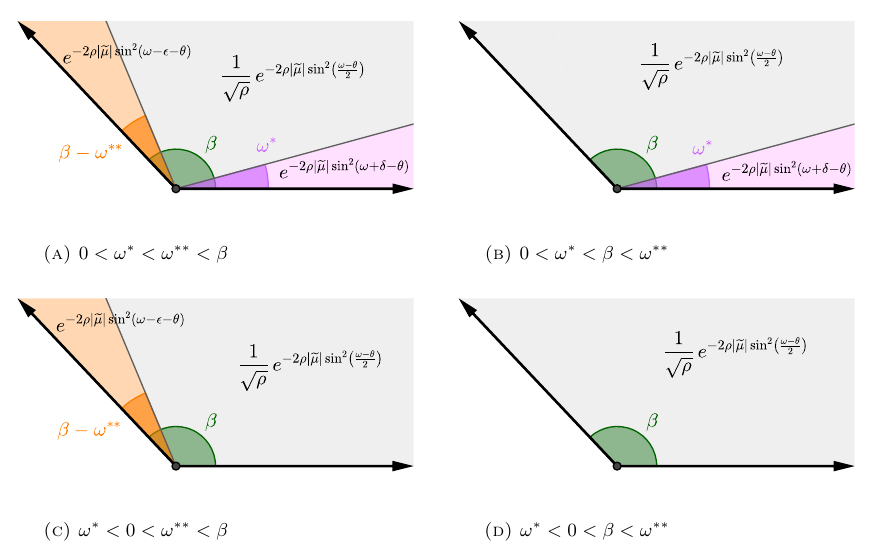}
\caption{Asymptotics of the Green function determined in Theorem~\ref{thm1} depending on the direction $\omega_0$: four different cases depending on the values of the angles $\omega^{*}=\theta-2\delta$ and $\omega^{**}=\theta+2\epsilon$. When $\omega_0$ lies in the gray region, the asymptotics are given by \eqref{eq:asymptsaddlepoint}, in the purple region by \eqref{eq:asymptpole1}, and in the orange region by \eqref{eq:asymptpole2}.}
\label{fig:coneasympt}
\end{figure}
The article \citeS{franceschi_kourkova_petit_asymptotics_2024} also contains theorems giving the precise asymptotic behaviour along the boundary, as well as in cases where the saddle point meets a pole.
The Martin boundary associated with this process and the corresponding harmonic functions can be computed from all these asymptotics.

\paragraph{Reflected Brownian motion in the Half-plane}

The article \citeS{franceschi_ernst_asymptotic_2021} determines the asymptotic behavior of the Green function of obliquely reflected Brownian motion in a half-plane. This article resolves an open question posed by Harrison in 2013 \cite{Web}. 

\paragraph{Killed space-time Brownian motion}

The article \cite{franceschi_spacetime_2024} studies a space-time Brownian motion with drift $\gamma\in (0,1)$,
\begin{equation}
B(t):=(t_0+t,y_0+W(t)+\gamma t),
\label{eq:B}
\end{equation}
where $W(t)$ is a standard Brownian motion, and $B$ is
killed (not reflected, this time!) at the boundary of the cone 
\begin{equation}
C:= \{ (t',y): 0<y<t' \}
\label{eq:C}
\end{equation}
which defines a moving boundary with two sides. We also define $T$ as the first exit time from the cone
\begin{equation}
T:=\inf \{t\geqslant 0 : B(t)\notin C \}
\label{eq:T}
\end{equation}
and the exit times on each edge of the cone
\begin{equation}
T_1:=\inf \{ t\geqslant 0 : B(t)=(t_0+t,t_0+t) , \ t>0 \}
\quad\text{and}\quad
T_2:=\inf \{ t\geqslant 0 : B(t)=(t_0+t,0) , \ t>0 \} .
\label{eq:T1T2}
\end{equation}
This article determines the parabolic Martin boundary and all the harmonic functions associated with this process. 
This model, with boundary conditions, cannot be handled using Girsanov's formula alone. The space-time Brownian motion conditioned to remain in a cone related to the root system of an affine Lie algebra has been studied, among others, by Defosseux \cite{Defosseux2016}.

\section{Infinite series of product forms via compensation approach}
\label{sec:compa}

The compensation approach originates from the works of \cite{adan91} and \cite{Adan_Wessels_Zijm_1993}, where it was developed for queueing problems and two-dimensional Markov processes, and has since proved particularly effective for the analysis of singular random walks~\cite{hoang_23}.

In our context, this method, which applies only in the degenerate case, constructs the stationary distribution or the Green's functions as an infinite series of product-form terms by iteratively compensating boundary errors. See \citeS{petit2024}, which focuses on Green's functions, and \citeS{franceschi_ichiba_karatzas_raschel_degenerate_2024}, which focuses on the invariant measure.

Here is an illustration of the method.
The invariant measure $\pi$ satisfies the following partial differential equation \cite{HaRe-81}
\begin{equation}
\begin{cases} {\mathcal{G}^*} \pi (u,v) =0  & (H_0),
\\
\partial_{{R^*_1}}\pi (0,v) -2\mu_1 \pi (0,v) =0, & (H_1),
\\
\partial_{{R^*_2}}\pi (u,0) -2\mu_2 \pi (u,0) =0 & (H_2).
\end{cases}
\label{eq:pdedeg}
\end{equation}
Solving this equation using the compensation approach leads to an explicit formula for the bivariate density of the invariant measure.
\begin{theorem}[Bivariate density] We have
$$\pi(u,v) =  C \sum_{n\geq 0} c_n e^{-a_n u-b_n v}+ C' \sum_{n\geq 0} c_n' e^{-a_n' u-b_n' v} . $$
All constants are made explicit in \citeS{franceschi_ichiba_karatzas_raschel_degenerate_2024}:
\begin{itemize}
   \item There is a recursive computation of the constants;
   \item $a_n,b_n,a_n',b_n'$ are quadratic polynomials in $n$, see Figure~\ref{fig:parabolecomp};
   \item $c_n=P_8(n)+(-1)^nQ_8(n)$
   and $c_n'=P_8'(n)+(-1)^nQ_8'(n),$
   \\ where $P_8$, $Q_8$, $P_8'$ and $Q_8'$ are polynomials of degree 8 in the symmetric case, and in the general case $$
c_n\underset{n\to\infty}{\sim} n^{2({-\gamma}+1)} .
$$
\end{itemize}
\end{theorem}

\begin{figure}[H]
\begin{center}
\includegraphics[width=0.43\textwidth]{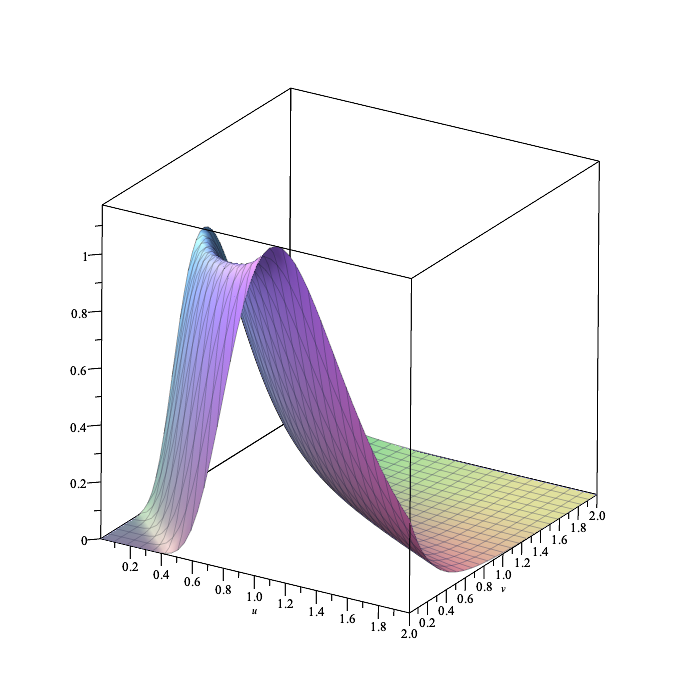}
\caption{Graph of the bivariate density}
\end{center}
\end{figure}

Here is a univariate illustration of this result.
For example, when $\gamma$ is a positive integer $m$, the density on the boundary equals to
$$\pi(u,0) = P_m\left(\frac{d}{du}\right) {\theta_{\mu_1}}(e^{-u})$$
where 
$${\theta_{\mu_1}}(q) := \sum_{n\in\mathbb Z}(n+\tfrac{\mu_1}{2}) q^{n(n+\mu_1)}$$ is a {{Jacobi theta-type function}} and $P_m$ an explicit polynomial of degree $m$.
In particular, in the case of symmetric collision $r_1=r_2=-1/2$, $\gamma=3$ we show that
 $$\pi(u,0)=\sum_{n\in\mathbb Z}(n-1)n(n+1)(n-1+\mu_1)(n+\mu_1)(n+1+\mu_1)(n+\tfrac{\mu_1}{2}) e^{-n(n+\mu_1)u}.$$

\begin{center}
\begin{figure}[h]
\includegraphics[width=0.9\textwidth]{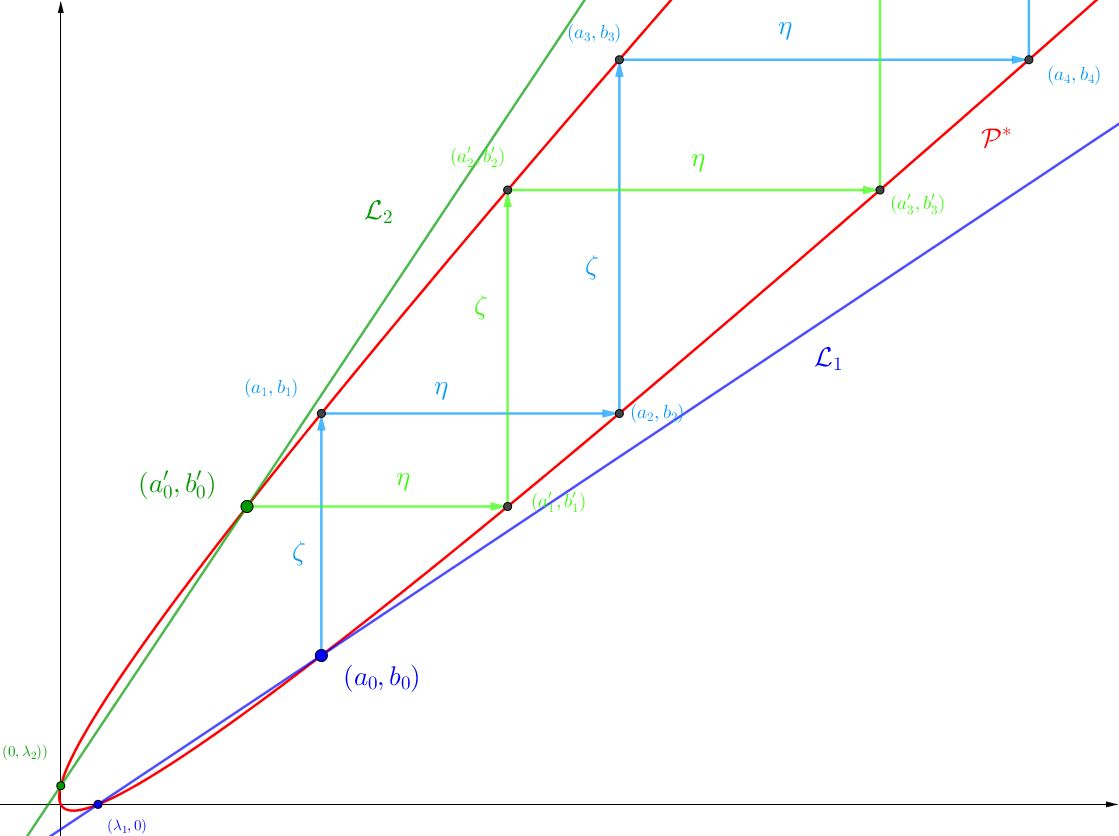}
\caption{Construction procedure of the constants $a_n,b_n,a_n',b_n'$}
\label{fig:parabolecomp}
\end{figure}
\end{center}

\paragraph{Heuristics of the compensation method}
Here are the main steps to obtain the previous result. We look for exponential functions satisfying $(H_0)$, $(H_1)$, and $(H_2)$. We have
\begin{align*}
e^{-au-bv}\text{ satisfies } (H_0)
\Longleftrightarrow
(a,b)\in \mathcal{P}^*:= \{(x,y)\in\mathbb{R}^2: (x-y)^2-2\mu_1x-2\mu_2y=0 \}
\end{align*}
and similarly
\begin{align*}
\nonumber
e^{-au-bv}\text{ satisfies } (H_1)
\Longleftrightarrow
(a,b)\in \mathcal{L}_1, \text{ where } 
\mathcal{L}_1:=\{(x,y)\in\mathbb{R}^2: 2x-3y-2\mu_1=0 \},
\end{align*}
\begin{align*}
\nonumber
e^{-a'u-b'v}\text{ satisfies } (H_2)
\Longleftrightarrow
(a',b')\in\mathcal{L}_2, \text{ with } 
\mathcal{L}_2:=\{(x,y)\in\mathbb{R}^2: 2y-3x-2\mu_2=0 \}.
\label{eq:L2}
\end{align*}
The parabola $\mathcal{P}^*$ and the lines $\mathcal{L}_1$ and $\mathcal{L}_2$ defined above can be visualized in Figure~\ref{fig:parabolecomp}, as well as the points $(a_0,b_0)\in \mathcal{P}^* \cap \mathcal{L}_1$ and $(a'_0,b'_0)\in \mathcal{P}^* \cap \mathcal{L}_2$.
We choose the constants $c_n$ and $c_n'$ to obtain the two unilateral compensations starting from $(a_0,b_0)$ and $(a_0',b_0')$:

\begin{equation*}
\label{eq:value_p(u,v)}
   p(u,v):=\overunderbraces{&\br{1}{\in H_1}& &\br{3}{\in H_1}& &\br{3}{\in H_1}}%
  {&c_0e^{-a_0u-b_0v} &+& c_1e^{-a_1u-b_1v}  &+&c_2e^{-a_2u-b_2v} &+&c_3e^{-a_3u-b_3v} &+ & c_4e^{-a_4u-b_4v} &+ \cdots}%
  {& \br{3}{\in H_2} & &\br{3}{\in H_2}} \in (H_0),
\end{equation*}
\begin{equation*}
\label{eq:value_pp(u,v)}
   p'(u,v):=\overunderbraces{&\br{3}{\in H_1}& &\br{3}{\in H_1}}%
  {&c'_0e^{-a'_0u-b'_0v} &+& c'_1e^{-a'_1u-b'_1v}  &+&c'_2e^{-a'_2u-b'_2v} &+&c'_3e^{-a'_3u-b'_3v} &+ & c'_4e^{-a'_4u-b'_4v} &+ \cdots}%
  {& \br{1}{\in H_2} & &\br{3}{\in H_2} & &\br{3}{\in H_2}} \in (H_0).
\end{equation*}
This approach raises several questions.
For all values of $C$ and $C'$, the linear combination
    \begin{equation*}
\label{eq:convex_combin_ppp}
   Cp(u,v)+C'p'(u,v)
\end{equation*}
is a solution to the partial differential equation \eqref{eq:pdedeg}. In this case, the compensation approach yields an infinite number of solutions. And we must determine which one is the correct one to adjust the constants $C$ and $C'$. This can be done using the previous results giving the boundary density. We can note that the functions $p$ and $p'$ may be negative, and we conjecture that there actually exists a unique choice of $C$ and $C'$ (up to a global scaling) such that $Cp+C'p'$ is non-negative.

In general, degeneracy is required to ensure convergence. If we had considered a non-degenerate Brownian motion, we would have obtained an ellipse instead of a parabola, and the sequence obtained by compensation would not necessarily have converged (since the $(a_n, b_n)$ belonging to the ellipse, which is bounded, would not tend to minus infinity).

In three or more dimensions, this method can be applied in cases of advanced degeneracy. Compensation is then performed by summing over a binary tree (or a tree of higher degree) rather than over $\mathbb{Z}$, as is the case in two dimensions. See, for example, the article by \citet{adan_wessels_zijm_92}.

\newpage
\appendix

\section{Scaling and normalization of the reflected Brownian motion}
\label{rescaling}

We describe a canonical normalization of the reflected Brownian motion in the quadrant obtained by diagonal scaling and time rescaling. This change of variables normalizes the covariance matrix, the drift, and the reflection matrix.

\begin{prop}[Canonical normalization of the reflected Brownian motion]
Let $(Z_t)_{t\ge 0}$ be a reflected Brownian motion in $\mathbb{R}_+^2$ satisfying
\[
Z_t = Z_0 + \mu t + B_t + R L_t,
\]
where the covariance matrix, the reflecting matrix and the drift are
\[
\Sigma =
\begin{pmatrix}
\sigma_{11} & \sigma_{12}\\
\sigma_{12} & \sigma_{22}
\end{pmatrix},
\qquad
R =
\begin{pmatrix}
r_{11} & r_{12}\\
r_{21} & r_{22}
\end{pmatrix},
\qquad
\mu= \begin{pmatrix}
\mu_1 \\
\mu_2 
\end{pmatrix}.
\]
Define
\begin{equation}
A =
\begin{pmatrix}
\sigma_{11}^{-1/2} & 0\\
0 & \sigma_{22}^{-1/2}
\end{pmatrix},
\qquad
c = |A\mu|= \sqrt{
\frac{\mu_1^2}{\sigma_{11}} + \frac{\mu_2^2}{\sigma_{22}}
},
\label{eq:Ac}
\end{equation}
and introduce the process
\[
\bar Z_t = c\, A\, Z_{t/c^2}.
\]
Then $(\bar Z_t)_{t\ge 0}$ is a reflected Brownian motion in $\mathbb{R}_+^2$ satisfying
\[
\bar Z_t = \bar Z_0 + \bar\mu t + \bar B_t + \bar R\, \bar L_t,
\]
where the covariance matrix, the reflecting matrix and the drift are
\[
\bar\Sigma =
\begin{pmatrix}
1 & \rho\\
\rho & 1
\end{pmatrix}
\text{ with }
\rho = \frac{\sigma_{12}}{\sqrt{\sigma_{11}\sigma_{22}}},
\quad
\bar R =
\begin{pmatrix}
1 &
\sqrt{\frac{\sigma_{22}}{\sigma_{11}}} \frac{ r_{12}}{r_{22}}\\[1em]
\sqrt{\frac{\sigma_{11}}{\sigma_{22}}} \frac{ r_{21}}{r_{11}} &
1
\end{pmatrix},
\quad
\bar\mu = \frac{A\mu}{|A\mu|}
\text{ with }
|\bar\mu|=1,
\]
and
\[
\bar L_t^i = \frac{cr_{ii} }{\sqrt{\sigma_{ii}}}\; L_{t/c^2}^i,
\qquad i=1,2,
\]
is the local time of the process on the axes.
\end{prop}
\begin{proof}
Set
\[
A=\begin{pmatrix}
\sigma_{11}^{-1/2} & 0\\
0 & \sigma_{22}^{-1/2}
\end{pmatrix},
\qquad
c=|A\mu|,
\qquad
D=\begin{pmatrix}
r_{11} & 0\\
0 & r_{22}
\end{pmatrix}.
\]
By definition,
\[
\bar Z_t
=cA Z_{t/c^2}=
cAZ_0+\frac{A\mu}{c}\,t+cA B_{t/c^2}+cA R L_{t/c^2}.
\]
By Brownian scaling, \(\bar B_t := cA B_{t/c^2}\) is a Brownian motion with covariance matrix
\[
A\Sigma A^\top
=
\begin{pmatrix}
1 & \rho\\
\rho & 1
\end{pmatrix},
\quad\text{where}\quad
\rho=\frac{\sigma_{12}}{\sqrt{\sigma_{11}\sigma_{22}}}.
\]
The drift becomes
\[
\bar\mu=\frac{A\mu}{|A\mu|}, \qquad |\bar\mu|=1.
\]
To conclude, it suffices to observe that
\[
cA R L_{t/c^2}
=
\underbrace{A R A^{-1}D^{-1}}_{=\bar R}
\underbrace{c D A L_{t/c^2}}_{=\bar L_t}.
\]
The rescaled local time vector is
\[
\bar L_t
=
cD A\, L_{t/c^2},
\]
that is,
\[
\bar L_t^i=\frac{c\, r_{ii}}{\sqrt{\sigma_{ii}}}\,L_{t/c^2}^i,
\qquad i=1,2.
\]
The new reflection matrix is
\[
\bar R=AR A^{-1}D^{-1}.
\]
A direct computation yields the expression of $\bar R$ given in the proposition.
This proves that \((\bar Z_t)_{t\ge0}\) is a reflected Brownian motion in \(\mathbb{R}_+^2\) with parameters \(\bar\Sigma\), \(\bar R\), \(\bar\mu\), and local time \(\bar L_t\).
\end{proof}

\begin{prop}[Transformation of the transition density]
Let $p_t(z_0,z)$ and $\bar p_t(\bar z_0,\bar z)$ denote the transition densities of $Z_t$ and $\bar Z_t$ with respect to the Lebesgue measure on $\mathbb{R}_+^2$. Then, for $A$ and $c$ defined in~\eqref{eq:Ac},
\[
p_t(z_0,z)
=
c^2 \det A\;
\bar p_{c^2 t}\!\left(cA z_0,\,cA z\right).
\]
Equivalently, for $z_0=(u,v)$ and $z=(z_1,z_2)\in\mathbb{R}_+^2$,
\[
p_t(z_0,z)
=
\frac{c^2}{\sqrt{\sigma_{11}\sigma_{22}}}\;
\bar p_{c^2 t}\!\left(
\left(\tfrac{c}{\sqrt{\sigma_{11}}} u,\, \tfrac{c}{\sqrt{\sigma_{22}}} v\right),
\left(\tfrac{c}{\sqrt{\sigma_{11}}} z_1,\, \tfrac{c}{\sqrt{\sigma_{22}}} z_2\right)
\right).
\]
\end{prop}

\begin{proof}
Recall that
\[
Z_t = (cA)^{-1} \bar Z_{c^2 t}.
\]
For any Borel set $B \subset \mathbb{R}_+^2$, we have
\[
\mathbb{P}\big(Z_t \in B \mid Z_0=z_0\big)
=
\mathbb{P}\Big(\bar Z_{c^2 t} \in cA B \mid \bar Z_0=cA z_0\Big).
\]
Using the transition density $\bar p_t(\bar z_0,\bar z)$, this yields
\[
\mathbb{P}(Z_t \in B \mid Z_0=z_0)
=
\int_{cA B} \bar p_{c^2 t}(cA z_0,\bar z)\, d\bar z.
\]
Performing the change of variables $\bar z = cA z$, we obtain
\[
d\bar z = c^2 \det A\, dz,
\]
and therefore
\[
\mathbb{P}(Z_t \in B \mid Z_0=z_0)
=
\int_B c^2 \det A\;
\bar p_{c^2 t}(cA z_0, cA z)\, dz.
\]
This proves the desired expression for $p_t(z_0,z)$. The second formula follows from the explicit form of $A$ and $\det A$.
\end{proof}

\paragraph{Acknowledgement}
This survey is a compact synthesis based on the cited works. I would like to express my deep gratitude to all my co-authors, without whom this work would not have been possible.

\paragraph{Funding}
This project has received funding from Agence Nationale de la Recherche, ANR JCJC programme under the Grant Agreement ANR-22-CE40-0002 (ANR RESYST).

\end{document}